\documentclass[12pt]{amsart}
\usepackage{amsmath}
\usepackage{amstext}
\usepackage{amsfonts}
\usepackage{amssymb}
\usepackage{amsthm}
\usepackage{amsrefs}
 \usepackage[normalem]{ulem}

\usepackage{microtype}
\usepackage{color,hyperref}
\definecolor{darkblue}{rgb}{0.0,0.0,0.3}
\hypersetup{colorlinks,breaklinks,
            linkcolor=red,urlcolor=red,
            anchorcolor=red,citecolor=red}

\usepackage{tikz}
\usetikzlibrary{arrows}
\usetikzlibrary{positioning}

\theoremstyle{plain}
\newtheorem{thm}{Theorem}[section]

\newtheorem{cor}[thm]{Corollary}
\newtheorem{prop}[thm]{Proposition}
\newtheorem{lem}[thm]{Lemma}
\newtheorem*{csconj}{Connes--Sullivan Conjecture (Zimmer's Theorem)}

\theoremstyle{definition}
\newtheorem{defn}[thm]{Definition}

\newtheorem{rem}[thm]{Remark}

\numberwithin{equation}{section}

\newcommand{\bC}{{\mathbb{C}}}

\newcommand{\bN}{{\mathbb{N}}}

\newcommand{\B}{{\mathcal{B}}}
\newcommand{\C}{{\mathcal{C}}}
\newcommand{\D}{{\mathcal{D}}}

\renewcommand{\P}{{\mathcal{P}}}
\renewcommand{\S}{{\mathcal{S}}}

\newcommand{\rcp}[2]{{#1} \rtimes_\mathrm{r} {#2}}

\newcommand{\ca}{\mathrm{C}^*}
\newcommand{\fb}{\partial_F G}
\newcommand{\aut}{\operatorname{Aut}}
\newcommand{\core}{\operatorname{core}}

\newcommand{\IB}{B}
\newcommand{\IN}{{\mathbb N}}
\newcommand{\IC}{{\mathbb C}}

\newcommand{\IR}{{\mathbb R}}

\newcommand{\cH}{{\mathcal H}}
\newcommand{\cK}{{\mathcal K}}
\newcommand{\cL}{{\mathcal L}}
\newcommand{\id}{\mathrm{id}}
\newcommand{\eps}{\varepsilon}
\DeclareMathOperator{\ran}{ran}
\DeclareMathOperator{\PGL}{PGL}
\newcommand{\ip}[1]{\mathopen{\langle}#1\mathclose{\rangle}}

\begin{document}

\title[C*-simplicity and the unique trace property]{C*-simplicity and the unique trace property for discrete groups}

\author[E. Breuillard]{Emmanuel Breuillard}
\address{Laboratoire de Math\'ematiques\\Universit\'e Paris-Sud 11\\
91405 Orsay cedex, France}
\email{emmanuel.breuillard@math.u-psud.fr}

\author[M. Kalantar]{Mehrdad Kalantar}
\address{Department of Mathematics\\University of Houston\\
Houston, TX, 77204-3008, United States}
\email{kalantar@math.uh.edu}

\author[M. Kennedy]{Matthew Kennedy}
\address{Department of Pure Mathematics\\University of Waterloo\\
Waterloo, ON, N2L 3G1, Canada}
\email{matt.kennedy@uwaterloo.ca}

\author[N. Ozawa]{Narutaka Ozawa}
\address{Research Institute for Mathematical Sciences\\Kyoto University\\
Kyoto 606-8502, Japan}
\email{narutaka@kurims.kyoto-u.ac.jp}

\begin{abstract}
A discrete group is said to be C*-simple if its reduced C*-algebra is simple, and is said to have the unique trace property if its reduced C*-algebra has a unique tracial state. A dynamical characterization of C*-simplicity was recently obtained by the second and third named authors. In this paper, we introduce new methods for working with group and crossed product C*-algebras that allow us to take the study of C*-simplicity a step further, and in addition to settle the longstanding open problem of characterizing groups with the unique trace property. We give a new and self-contained proof of the aforementioned characterization of C*-simplicity. This yields a new characterization of C*-simplicity in terms of the weak containment of quasi-regular representations. We introduce a convenient algebraic condition that implies C*-simplicity, and show that this condition is satisfied by a vast class of groups, encompassing virtually all previously known examples as well as many new ones. We also settle a question of Skandalis and de la Harpe on the simplicity of reduced crossed products. Finally, we introduce a new property for discrete groups that is closely related to C*-simplicity, and use it to prove a broad generalization of a theorem of Zimmer, originally conjectured by Connes and Sullivan, about amenable actions.
\end{abstract}

\subjclass[2010]{Primary 46L35; Secondary 20F65, 37A20, 43A07}
\keywords{group actions, Furstenberg boundary, reduced C*-algebra, simplicity, unique trace}
\thanks{First author supported by ERC Grant Number 617129.}
\thanks{Third author supported by NSERC Grant Number 418585.}
\thanks{Fourth author supported by JSPS KAKENHI Grant Number 26400114.}
\maketitle

\setcounter{tocdepth}{1}

\tableofcontents

\section{Introduction}

Let $G$ be a discrete group. Recall that the reduced C*-algebra $\ca_r(G)$ of $G$ is the norm closure of the algebra of operators on $\ell^2(G)$ generated by the left regular representation $\lambda_G$ of $G$. The group $G$ is said to be {\em C*-simple} if $\ca_r(G)$ is simple, meaning that the only norm-closed two-sided ideals in $\ca_r(G)$ are zero and $\ca_r(G)$ itself. Recall further that a trace on $\ca_r(G)$ is a tracial state, i.e. a unital positive $G$-invariant linear functional on $\ca_r(G)$. The group $G$ is said to have the {\em unique trace property} if $\ca_r(G)$ has a unique trace, namely the canonical trace $\tau_\lambda$ defined by $\tau_\lambda(a):=\langle a\delta_e,\delta_e\rangle$ for $a \in \ca_r(G)$.

Since Powers' proof \cite{P1975} in 1975 that the free group on two generators is both C*-simple and has the unique trace property, it had been a major open problem to characterize groups with either of these properties, and in particular to determine whether they are equivalent (see e.g. \cite{D2007} for this fact, and for a nice general survey of the subject matter).

It has long been recognized that the simplicity of the reduced C*-algebra $\ca_r(G)$, and more generally of reduced crossed product C*-algebras of the form $C(X) \rtimes_r G$, where $X$ is a compact $G$-space, is related to the topological dynamics of the $G$-action on $X$ (see in particular the work of Kawamura and Tomiyama \cite{KT1990}, and the work of Archbold and Spielberg \cite{AS1994}).

Recently, the second and third named authors \cite{KK2014}*{Theorem 6.2} established that the dynamical properties of the Furstenberg boundary $\fb$ of $G$, a compact $G$-space that is well known to researchers in dynamics, completely determines whether $G$ is C*-simple.

\begin{thm}[\cite{KK2014}] \label{thm:c-star-simplicity}
A discrete group is C*-simple if and only if its action on the Furstenberg boundary $\partial_F G$ of $G$ is free.
\end{thm}

In this paper, we introduce new methods for working with group and crossed product C*-algebras. We take a dynamical point of view, beginning with a preliminary  study of the Furstenberg boundary. In particular, we obtain a new and self-contained proof of Theorem~\ref{thm:c-star-simplicity}, which proceeds essentially from first principles, and requires no advanced operator-algebraic material. This new proof yields as a by-product the following new characterization of C*-simplicity:

\begin{thm}\label{thm:weakly-contained}
A discrete group $G$ is C*-simple if and only if for every amenable subgroup $H \leq G$, the quasi-regular representation $\lambda_{G/H}$ is weakly equivalent to the left regular representation $\lambda_G$.
\end{thm}

Recall (see e.g. \cite{Di1977}*{Section 3.4.4} or \cite{D2007}*{Section 7}) that a discrete group $G$ is C*-simple if and only if every unitary representation that is weakly contained in $\lambda_G$ is weakly equivalent to $\lambda_G$. The above result says that it is enough to consider quasi-regular representations with respect to amenable subgroups.

In turn, these new methods allow us to completely settle the problem of characterizing groups with the unique trace property.

\begin{thm} \label{thm:unique-trace-property}
A discrete group has the unique trace property if and only if its amenable radical is trivial. In particular, every C*-simple group has the unique trace property.
\end{thm}

The proofs of these results occupy the first three sections of this paper, and are remarkably short and self-contained. The last result is also used (in combination with an observation from \cite{T2012}*{Theorem 5.14}) to obtain another proof of the fact, recently proved in \cite{BDL2014}, that amenable invariant random subgroups of a discrete group concentrate on the amenable radical.

The remainder of the paper is devoted to a thorough investigation of C*-simplicity. First we show that C*-simplicity is stable under group extensions, answering a question of de la Harpe and Pr\'eaux \cite{DP2011}*{Section 2, Question (Q)}.

\begin{thm}\label{thm:c-star-simplicity-closed-wrt-extensions}
Let $G$ be a discrete group and let $N \le G$ be a normal subgroup. Then $G$ is C*-simple if and only if both $N$ and $C_G(N)$ are C*-simple, where $C_G(N)$ denotes the centralizer of $N$ in $G$. In particular, C*-simplicity is closed under extension.
\end{thm}

We then introduce a convenient algebraic condition, the absence of amenable \emph{normalish subgroups}, that implies C*-simplicity (see Theorem \ref{thm:not-c-star-simple-implies-normalish}). This criterion is easy to work with, and allows us to give a short proof of the C*-simplicity of a large class of groups that encompasses virtually all previously known examples (e.g. \cites{BCH1994, P2008, OO2014}) as well as many new ones. In particular, we prove:

\begin{thm}\label{thm:betti}
A discrete group with trivial amenable radical having either non-trivial bounded cohomology or non-vanishing $\ell^2$-Betti numbers is C*-simple.
\end{thm}

We also recover the following result from \cite{P2008}. Recall that a linear group is a subgroup of $\mathrm{GL}_d(K)$ for some field $K$.

\begin{thm}\label{thm:lineargroups}
A linear group is C*-simple if and only if its amenable radical is trivial.
\end{thm}

Using a different argument, we also handle the following class of groups:

\begin{thm}\label{thm:few}
A discrete group with only countably many amenable subgroups is C*-simple if and only if its amenable radical is trivial.
\end{thm}

The last result applies, for example, to torsion and torsion-free Tarski monster groups, as well as to free Burnside groups $B(m,n)$ for $m \geq 2$ and $n$ odd and sufficiently large. Thus we recover a recent result of Osin and Olshanski \cite{OO2014} that yields in particular the existence of C*-simple groups without non-abelian free subgroups (see also \cite{LM2016} for another example).

We emphasize that, prior to our work, essentially the only method available for proving C*-simplicity was the method introduced by Powers in his 1975 paper \cite{P1975}, which consists of proving a certain spectral radius estimate, and which often requires exhibiting ``ping-pong partners'' for certain boundary actions of the group. As the case of linear groups exemplifies, this method is usually difficult to implement in practice. Our \emph{no amenable normalish subgroup} criterion (Theorem \ref{thm:not-c-star-simple-implies-normalish}) leads to a considerably simplified analysis in each example. We note however, that there is no conceptual limitation to Powers' method, since we now know that Powers' criterion is in fact equivalent to C*-simplicity \cites{K2015, H2015}.

Further in the paper, we investigate reduced crossed products. We show that the reduced crossed product over a C*-simple group is simple whenever the underlying C*-algebra has no invariant closed ideals. This answers a question of de la Harpe and Skandalis \cite{DS1986}*{page 242}, and applies in particular to the C*-algebra of commutative functions on a compact space equipped with a minimal group action.

\begin{thm}
Let $G$ be a discrete C*-simple group. For any unital $G$-C*-algebra $A$ having no non-trivial $G$-invariant closed ideal, the reduced crossed product $\rcp{A}{G}$ is simple.
\end{thm}

We also recover and generalize an observation of Haagerup and Olesen \cite{HO2014} relating the amenability of Thompson's group $F$ to the C*-simplicity of Thompson's group $T$.

\begin{prop}\label{prop:amen-stab} If $X$ is a $G$-boundary such that the point stabilizer $G_x$ is amenable for some $x\in X$, then $G$ is C*-simple if and only if $X$ is topologically free.
\end{prop}

We then further study the relation between the C*-simplicity of $G$ and the size of point stabilizers of boundary actions. In particular, we show that if the stabilizer of some point on a $G$-boundary is C*-simple, then $G$ is C*-simple as well (see Proposition \ref{prop:G-not-c-star-simple-implies-G-x-not-c-star-simple}).

Finally, in the last section, we introduce a new operator-algebraic property of discrete groups that we call the Connes-Sullivan property, or property (CS) for short. It is a stronger property than C*-simplicity, in the sense that every group with property (CS) and trivial amenable radical is C*-simple.

\begin{defn}
A discrete group $G$ is said to have the {\em Connes-Sullivan property}, or property (CS) for short, if for every unitary representation $\pi : G \to B(\cH)$ weakly contained in the left regular representation of $G$ there exists a neighborhood $U$ of the identity in $B(\cH)$ such that $\pi^{-1}(U)$ belongs to the amenable radical $R_a(G)$ of $G$.
\end{defn}

We establish the following result:

\begin{thm}\label{thm:cs}
Let $G$ be a discrete group. If $G$ is linear or if the bounded cohomology of $G$ satisfies $H_b^2(G,\ell^2(G/R_a(G))) \ne 0$, then $G$ has property (CS).
\end{thm}

The Connes-Sullivan conjecture, which was proved by Zimmer \cite{Z1987}, asserts that subgroups of a Lie group that act amenably on it must be dense in a Lie subgroup with solvable identity component. The above theorem turns out to be a broad generalization of Zimmer's theorem, in the sense that Zimmer's theorem follows easily from the fact that linear groups have property (CS).

In addition to the introduction, this paper has seven other sections. In Section \ref{sec:fust} we recall the notion of Furstenberg boundary $\fb$ for a discrete group $G$ and we establish some of its basic topological properties, which we then translate into operator algebraic properties of the C*-algebra of continuous functions $C(\fb)$.  In Section \ref{sec:kk} we give a self-contained proof of Theorem \ref{thm:c-star-simplicity} and establish Theorem \ref{thm:weakly-contained}.  Section \ref{sec:uniqueness-of-trace} is devoted to the unique trace property and Section \ref{stability} to Theorem  \ref{thm:c-star-simplicity-closed-wrt-extensions}. Normalish subgroups are introduced in Section \ref{sec:examples}, where Theorems \ref{thm:betti}, \ref{thm:lineargroups} and \ref{thm:few} are proven. Finally,  Section \ref{sec:cross} is devoted to reduced cross-products, while the last section deals with property (CS) and Theorem \ref{thm:cs}.

\subsection*{New developments}

Several new developments have occurred since the first draft of this paper appeared in October 2014. First, in a dramatic turn of events, Le Boudec \cite{L2015 } exhibited the first examples of groups that have trivial amenable radical (and hence unique trace by Theorem \ref{thm:unique-trace-property}) but are not C*-simple. The existence of such an example, whose proof utilizes Proposition \ref{prop:amen-stab} to disprove C*-simplicity, has put an end to a longstanding question and a posteriori justifies the relevance of the various sufficient criterions for C*-simplicity expounded in the present paper and elsewhere.

Subsequently, Kennedy \cite{K2015} and Haagerup in a posthumous preprint \cite{H2015} independently discovered a new characterization of C*-simplicity. They showed that a group is C*-simple if and only if its reduced C*-algebra satisfies an averaging property similar to the one used by Powers in \cite{P1975}. Additionally, in \cite{K2015} a group is shown to be C*-simple if and only if it has no non-trivial amenable uniformly recurrent subgroups in the sense of Glasner and Weiss. This latter criterion has been recently used by Le Boudec and Matte Bon \cite{LM2016} to study the C*-simplicity of various groups of homeomorphisms and to show that Thompson's group $V$ is C*-simple, while the C*-simplicity of $T$ is equivalent to the non-amenability of $F$.

Bryder and Kennedy \cite{BK2016} studied the ideal structure of (twisted) crossed products over C*-simple groups. In particular, they established a bijective correspondence between maximal ideals of the reduced crossed product and maximal invariant ideals of the underlying C*-algebra. Finally, recent work of Raum explores the C*-simplicity of non-discrete groups \cites{R2015a,R2015b}, and very recent work of Ivanov and Omland \cite{IO2016} provides further examples of non-C*-simple groups, among other things.

\section{The Furstenberg boundary}\label{sec:fust}

In this section we recall the notion of the Furstenberg boundary of a discrete group. We also give a new direct proof of two important properties: its extremal disconnectedness and the amenability of point stabilizers, which were first proved in \cite{KK2014} using operator-algebraic techniques.

\subsection{Definitions}Let $G$ be a discrete group and $X$ a compact\footnote{In this paper, compact spaces are assumed to be Hausdorff.} topological space. A $G$-action on $X$ is a group homomorphism from $G$ to the group of homeomorphisms of $X$. A $G$-map between two compact $G$-spaces is a continuous $G$-equivariant map. The $G$-action on $X$ is said to be \emph{minimal} if the $G$-orbit $Gx$ is dense for every $x\in X$. It is said to be \emph{proximal} if for every pair $x,y\in X$ there is a net $t_i \in G$ such that $\lim t_i x=\lim t_i y$. It is said to be \emph{strongly proximal} if the induced $G$-action on the space $\mathcal{P}(X)$ of probability measures on $X$, is proximal. Furstenberg \cite{F1973} (see also \cite{G1976}) introduced the following notion.

\begin{defn}[$G$-boundary] A compact Hausdorff $G$-space $X$ is called a \emph{$G$-boundary} if it is minimal and strongly proximal.
\end{defn}

Boundary actions arise in many natural geometric contexts, beginning with the classical example of non-elementary discrete subgroups of $\PGL_2(\IR)$ acting on the projective line. But, as Furstenberg observed, they arise whenever one has an affine $G$-action on a compact convex space.

\begin{prop}[\cite{G1976}*{Theorem III.2.3}]\label{furst1} Suppose $G$ acts by affine maps on a locally convex topological vector space, and let $K$ be a $G$-invariant compact convex subset. Then $K$ contains a $G$-boundary $X$. It is unique if $K$ is irreducible. Conversely, if $K$ is the closed convex hull of $X$, then $K$ is irreducible.
\end{prop}

Irreducibility for $K$ means that there is no proper $G$-invariant compact convex subspace. The converse part is not stated explicitly in \cite{G1976}, but is clear from the proof. In fact, if $K$ is irreducible, then $X$ is the closure of the extreme points of $K$.

Every $G$-boundary $X$ arises in this way, because $X$ can be identified with the point masses in the space $\mathcal{P}(X)$ of probability measures on $X$. We will always view $X$ as a subspace of $\mathcal{P}(X)$ in this way.

Furstenberg showed \cite{F1973} that every group $G$ admits a universal boundary $\fb$, i.e. a boundary such that every $G$-boundary is a continuous $G$-equivariant image of $\fb$. Moreover he showed that it is unique up to $G$-equivariant homeomorphism. It is now called \emph{the Furstenberg boundary} $\fb$ of $G$. The uniqueness is a consequence of the following important property:

\begin{prop}[\cite{F1973}*{Proposition 4.2}]\label{furst2} Every $G$-map from a compact $G$-space $Y$ into $\mathcal{P}(X)$, where $X$ is a $G$-boundary, must contain $X$ in its range. If $Y$ is minimal, then there is at most one such map.
\end{prop}

\begin{proof} We may clearly assume that $Y$ is minimal for the first assertion. Then it is enough to show that two such maps will take the same value at some point of $Y$, and that this value is a point mass. If $\phi_1,\phi_2: Y \to X$ are $G$-maps, then $\mu_y:=\frac{1}{2}(\delta_{\phi_1(y)} + \delta_{\phi_2(y)})$ is a probability measure on $X$. Strong proximality implies that the $G$-orbit of $\mu_y$ contains a point mass in its closure, which ends the proof.
\end{proof}

In particular, the only $G$-map between $X$ and $\mathcal{P}(X)$ is the identity map $x \mapsto \delta_x$. An important immediate consequence of Proposition \ref{furst1} and the universal property of $\fb$ is the existence of \emph{boundary maps}. Namely given any compact $G$-space $X$, there exists a boundary map, i.e. a $G$-equivariant continuous map
\begin{equation}\label{bdy-map} b: \fb \to \mathcal{P}(X).\end{equation}

\subsection{Extremal disconnectedness}
Recall that a topological space is called \emph{extremally disconnected} (or Stonean) if the closure of every open set is open.

Gleason \cite{gleason} proved that a compact space $X$ is extremally disconnected if and only it is a projective object in the category of compact spaces with continuous maps as morphisms. This means that given $Y,Z$ compact spaces and continuous maps $a:X \to Y$ and $p:Z \twoheadrightarrow Y$ with $p$ surjective, there is a map $c:X \to Z$ with $a=p \circ c$.

\begin{prop}\label{stonean}The Furstenberg boundary $\fb$ of a discrete group $G$ is extremally disconnected.
\end{prop}

\begin{proof} We adapt an argument of Gleason \cite{gleason}*{Theorem 1.2}. Let $U$ be an open subset of $\fb$ and let $Y$ be the compact subset of $\fb \times \{0,1\}$ obtained by taking the disjoint union of $\overline{U} \times \{0\}$ and $U^c \times \{1\}$. Pick $x_0 \in \fb$ and define $\phi:G \to Y$ by $\phi(g)= (gx_0,0)$ if $gx_0 \in U$ and $\phi(g)=(gx_0,1)$ otherwise. By the universal property of the Stone-Cech compactification $\beta G$ of $G$, $\phi$ extends to a continuous map from $\beta G$ to $Y$ that we continue to denote by $\phi$. By $(\ref{bdy-map})$ there is a boundary map $b: \fb \to \mathcal{P}(\beta G)$. For $x \in \fb$, let $\mu_x=\phi_* \circ b(x) \in \mathcal{P}(Y)$. By Proposition  \ref{furst2}, the push-forward of $\mu_x$ on $\fb$ via the projection onto the first factor is the point mass $\delta_x$. In other words, $\mu_x$ is supported on $\{x\} \times \{0,1\}$. Hence $\mu_x(U^c \times \{1\}) = 0$ if $x \in U$, and $\mu_x(U^c \times \{1\}) = 1$ if $x \notin \overline{U}$. Since the map $x \mapsto \mu_x(U^c \times \{1\})$ is continuous, this implies that $\overline{U}$ is open.
\end{proof}

This proposition was first observed in \cite{KK2014}*{Remark 3.16} by other means. Applying Gleason's theorem \cite{gleason}, we conclude that $\fb$ is projective in the category of compact spaces with continuous maps. We now prove a $G$-equivariant analogue of this property, which also follows easily from Furstenberg's Propositions \ref{furst1} and \ref{furst2}.

\begin{prop}\label{lift} Let $K$ and $K'$ be compact convex $G$-spaces and let $p:K' \twoheadrightarrow K$ be a surjective affine $G$-map. Then any $G$-map $a: \fb \to K$ lifts to a $G$-map $c: \fb \to K'$, i.e. $a=p \circ c$.
\end{prop}

\begin{proof} Note that $a(\fb)$ is a $G$-boundary. Its convex hull $C$ is a compact convex $G$-invariant subspace of $K$, and hence so is $p^{-1}(C)$. Thus by Proposition \ref{furst1}, $p^{-1}(C)$ contains a $G$-boundary $X$, $C$ is irreducible and $a(\fb)$ is the unique $G$-boundary in $C$. It follows that $p(X)=a(\fb)$. By the universal property of $\fb$, there is a $G$-map $c:\fb \to X$. Proposition \ref{furst2} implies that $a$ and $p\circ c$ coincide, as desired.
\end{proof}

\subsection{Amenability of stabilizers}
A point $y \in \fb$ gives rise to an injective $G$-equivariant unital positive linear map $\sigma_{y}: C(\fb) \hookrightarrow \ell^\infty(G)$ defined by $\sigma_y(f)(g) = f(gy)$ for $f \in C(\fb)$ and $g \in G$. The positivity of $\sigma_y$ means that it sends non-negative functions to non-negative functions.

Dual to the existence of the boundary map $(\ref{bdy-map})$ is the existence of a $G$-equivariant unital positive linear map $\beta : C(X) \to C(\fb)$ defined by $\beta(f)(y) = \int_X f\, db(y)$ for $f \in C(X)$ and $y \in \fb$. Together with Proposition \ref{furst2}, this yields the following result.

\begin{lem}
There is a $G$-equivariant unital positive retraction $r:\ell^\infty(G) \to C(\fb)$, i.e. $r \circ \sigma_y = id_{|C(\fb)}$ for every $y \in \fb$.
\end{lem}

\begin{proof} Note that $\ell^\infty(G)=C(\beta G)$, where $\beta G$ is the Stone-Cech compactification of $G$. The dual $\beta : \fb \to \mathcal{P}(\beta G)$ of the boundary map yields the desired retraction.
\end{proof}

We obtain the following result as a direct consequence.

\begin{prop}\label{amen} For every $x \in \fb$, the point stabilizer $G_x=\{g \in G : gx=x\}$ is amenable.
\end{prop}

\begin{proof} Let $e_x:C(\fb) \to \IC$ denote the evaluation map at $x$. Since $G$ is discrete, $\ell^\infty(G_x)$ embeds naturally into $\ell^\infty(G)$. Given the retraction $r$ from the previous lemma, the composition $e_x \circ r$ is a $G_x$-invariant unital positive linear functional on $\ell^{\infty}(G_x)$, as desired.
\end{proof}

This also yields a quick proof, for discrete groups, of the following result of Furman \cite{F2003}. Recall that the amenable radical $R_a(G)$, introduced by Day, \cite{D1957} is the largest amenable normal subgroup of $G$.

\begin{prop}\label{furman} The kernel of the action of $G$ on $\fb$ coincides with the amenable radical of $G$.
\end{prop}
\begin{proof} Being amenable, $R_a(G)$ fixes a probability measure $\mu$ on $\fb$. Also, being normal it fixes $g\mu$ for every $g \in G$. Since $\fb$ is strongly proximal, $R_a(G)$ must fix a point, and hence must fix every point by the minimality of $\fb$. The converse follows immediately from Proposition \ref{amen}.
\end{proof}

Consequently we see that $\fb$ is reduced to a point if and only if $G$ is amenable \cite{G1976}, and that $\partial_F( G/R_a(G)) = \fb$. Also, it is easy to see that every strongly proximal compact $G$-space which is not reduced to a point, has no isolated points. In particular, this is the case for $\fb$. Finally, unless reduced to a point, extremally disconnected spaces are not second countable \cite{gleason}*{Theorem 1.3}. In particular, if $G$ is non-amenable, then $\fb$ is not metrizable and hence $C(\fb)$ is not separable.

\subsection{Injectivity, rigidity and essentiality}
The two important properties of $\fb$ proved by Furstenberg and recalled in Proposition \ref{furst1} and Proposition \ref{furst2} dualize to yield crucial operator-algebraic properties of the commutative C*-algebra $C(\fb)$.

Recall that a $G$-C*-algebra $A$ is a unital C*-algebra endowed with a $G$-action by automorphisms. A state of $A$ is a unital positive functional, and the state space $S(A)$ is the compact convex space consisting of all states on $A$. A linear map between C*-algebras is positive if it sends positive elements to positive elements, and it is unital if it sends the unit to the unit.

If $X$ and $Y$ are compact spaces, then there is a correspondence between unital positive linear maps $\phi : C(X) \to C(Y)$ and continuous maps $\widetilde{\phi} : Y \to \mathcal{P}(X)$ given by the rule
\[
\phi(f)(y) = \int_X f d\widetilde{\phi}(y), \quad f \in C(X),\ y \in Y.
\]
It is clear that this correspondence is $G$-equivariant if $X$ and $Y$ are $G$-spaces. It is also clear that the map $\phi$ is an isometric embedding if $\widetilde{\phi}$ contains all the point masses in its image. The situation is analogous if $C(Y)$ is replaced by an arbitrary $G$-C*-algebra $A$, and $Y$ is replaced by the state space $S(A)$ of $A$. Consequently, the following result is an immediate translation of Proposition \ref{furst2}.

\begin{prop}[$G$-rigidity and $G$-essentiality]\label{rigidity} The identity map is the only $G$-equivariant unital positive linear map from $C(\fb)$ to itself. Moreover every $G$-equivariant unital positive linear map from $C(\fb)$ to a unital $G$-C*-algebra is an isometric embedding.
\end{prop}

Note that this isometric embedding need not be an algebra homomorphism. This proposition and the next were first obtained in \cite{KK2014} as consequences of Hamana's theory of injective envelopes of C*-algebras. In this language, the first assertion above means that $C(\fb)$ is a $G$-rigid extension of $\bC$, while the second means that $C(\fb)$ is a $G$-essential extension of $\bC$. One of the key findings of that paper was the realization that the Hamana boundary of a discrete group is the same thing as the Furstenberg boundary. In the current text, we reverse this point of view and start with the topological properties of the Furstenberg boundary in order to deduce the desired operator algebraic results.

Along the same lines, Proposition \ref{lift} immediately yields the following result.

\begin{prop}[$G$-injectivity]\label{injectivity} The algebra $C(\fb)$ is $G$-injective. Namely, given two unital $G$-C*-algebras $A \subset B$, any $G$-equivariant unital positive map $A \to C(\fb)$ lifts to a $G$-equivariant unital positive  map $B \to C(\fb)$.
\end{prop}

\begin{proof} Apply Proposition \ref{lift} to the state spaces $K=S(A)$ and $K'=S(B)$.
\end{proof}

\begin{rem}\label{op-sys} Proposition \ref{injectivity} also holds (with the same proof) assuming only that $A$ and $B$ are $G$-operator systems, i.e. self-adjoint subspaces of a unital C*-algebra containing the unit and endowed with a $G$-action by unital complete order isomorphisms, i.e. maps that preserve the matrix order structure (see \cite{KK2014}).
\end{rem}

\section{Two characterizations of C*-simplicity}\label{sec:kk}
In this section, we establish two characterizations of C*-simplicity. The first (Theorem \ref{thm:c-star-simplicity-2}) is due to Kalantar and Kennedy \cite{KK2014}. Our goal here will be to give a new, entirely self-contained proof. As a payoff, this will yield a second characterization (Theorem \ref{newchar} below), that is new. Recall that a compact $G$-space $X$ is called \emph{topologically free}, if the set of fixed points of each non trivial element $g \in G$ has empty interior.

\begin{thm} \label{thm:c-star-simplicity-2}
Let $G$ be a discrete group. The following are equivalent:
\begin{enumerate}
\item $G$ is C*-simple,
\item  $G$  acts freely on its Furstenberg boundary $\fb$,
\item there exists a topologically free $G$-boundary.
\end{enumerate}
\end{thm}

The Furstenberg boundary is not easy to describe concretely. On the other hand, $G$-boundaries often appear naturally in geometry. So in practice, (3) is typically used to prove C*-simplicity.

We split the proof of this theorem into three independent parts.

\subsection{Proof of Theorem \ref{thm:c-star-simplicity-2} ((2) $\Leftrightarrow$ (3))}

Since (2) implies (3) by definition, we need to show that if $G$ admits a topologically free boundary $X$, then it acts freely on $\fb$. Lemma \ref{inte} below implies that $G$ acts topologically freely on $\fb$, and Lemma \ref{topfree} that it acts freely.

\begin{lem}\label{inte} A $G$-map $\pi:Y \to X$ between two minimal compact $G$-spaces sends open sets to sets of non-empty interior.
\end{lem}

\begin{proof} If $U$ is open in $Y$, the minimality of $Y$ implies that it is covered by all translates of $U$, and hence by finitely many since $Y$ is also compact. The image of these translates under $\pi$ covers $X$, hence they have non-empty interior.
\end{proof}

\begin{lem}\label{topfree} The $G$-action on $\fb$ is free if and only if it is topologically free.
\end{lem}

\begin{proof} Recall that $\fb$ is extremally disconnected (Proposition \ref{stonean}). But the fixed point set of any homeomorphism of an extremally disconnected space is open (see \cite{F1971} or \cite{P2012}*{Proposition 2.7}).
\end{proof}

\subsection{Proof of Theorem \ref{thm:c-star-simplicity-2} ((1) $\Rightarrow$ (3))}  The proof will follow easily from the following proposition.

\begin{prop} \label{prop:not-weakly-contained-quasi-regular}
Let $G$ be a non-trivial discrete group and let $X$ be a $G$-boundary. Suppose there is $s \in G \setminus \{e\}$ such that the set $X_s = \{x \in X \mid sx = x\}$ of fixed points of $s$ has non-empty interior. Then the left regular representation $\lambda_G$ is not weakly contained in the quasi-regular representation $\lambda_{G/G_x}$ corresponding to $G_x$.
\end{prop}

\begin{rem}\label{weak-cont-def} Recall that a discrete group $G$ is C*-simple if and only if the following property holds: Every unitary representation of $G$ which is weakly contained in the regular representation of $G$ is weakly equivalent to it (for a proof see \cite{Di1977}*{Section 3.4.4} or \cite{D2007}*{Section 7}).
\end{rem}

\begin{proof}[Proof of (1) $\Rightarrow$ (3) in Thm. \ref{thm:c-star-simplicity-2}] Apply the proposition to $X=\fb$. If the action is not topologically free, then $\lambda_G$ is not weakly contained in $\lambda_{G/G_x}$. However,  $\lambda_{G/G_x}$ is weakly contained in $\lambda_{G}$, because $G_x$ is amenable by Proposition \ref{amen}. Hence $G$ is not C*-simple.\end{proof}

We now focus on the proof of Proposition \ref{prop:not-weakly-contained-quasi-regular}. We require the following lemma.

\begin{lem} \label{lem:not-weakly-contained-1}
Let $G$ be a non-trivial discrete group and let $X$ be a $G$-boundary. For every non-empty open subset $U \subset X$ and $\varepsilon > 0$, there is a finite subset $F \subset G \setminus \{e\}$ such that for every probability measure $\mu$ on $\fb$, there is $t \in F$ satisfying $\mu(t U) > 1 - \varepsilon$.
\end{lem}

\begin{proof}
Fix $x \in U$. For each probability measure $\mu$ on $X$, strong proximality and minimality implies there is $t_\mu \in G \setminus \{e\}$ such that
\[
1 - \mu(t_\mu U) = (\delta_x - t_\mu^{-1} \mu)(U) < \varepsilon.
\]
By continuity there is an open neighborhood $V_\mu$ of $\mu$ in the compact space $\P(X)$ of probability measures on $X$ such that $1 - \nu(t_\mu U) < \varepsilon$ for every $\nu \in V_\mu$. The open sets $V_\mu$ cover $\P(X)$, so by compactness there is a finite subcover $V_{\mu_1},\ldots,V_{\mu_n}$, and we may take $F = \{t_{\mu_1},\ldots,t_{\mu_n}\}$.
\end{proof}

\begin{proof}[Proof of Proposition \ref{prop:not-weakly-contained-quasi-regular}]
If $\xi= \delta_e \in \ell^2(G)$ denotes the point mass at $e$, then the matrix coefficient $\langle \lambda_{G}(g)\xi, \xi \rangle$ is zero if $g \neq e$. We are going to show that it cannot be approximated by convex combinations of matrix coefficients of $\lambda_{G/G_x}$.
To this end, let $U$ be the interior of the fixed point set $X_s$, fix $\eps=1/3$ and let $F \subset G$ be the finite set given by Lemma \ref{lem:not-weakly-contained-1}. Assume, by way of contradiction, that there exist finitely many unit vectors $\xi_1,\ldots,\xi_m$ in $\ell^2(G/G_x)$ such that
\begin{equation}
\label{weak-cont} \left| \langle \lambda_{G}(g)\xi, \xi \rangle - \frac{1}{m} \sum_{j=1}^m \langle \lambda_{G/G_x}(g)\xi_j, \xi_j \rangle \right| < \eps,
\end{equation}
for every $g$ in the finite set $\{tst^{-1} : t \in F\}$.  Define probability measures on $X$ by
\[
\mu:= \frac{1}{m} \sum_{j=1}^m \mu_j, \quad \mu_j:= \sum_{y \in G\cdot x} |\xi_j(y)|^2 \delta_y,
\]
where we have identified the orbit $G\cdot x$ with $G/G_x$. By Lemma \ref{lem:not-weakly-contained-1} there is $t \in F$ such that $\mu(tU^c)<\eps$. Hence
\begin{equation}\label{muest}
\frac{1}{m}\sum_{j=1}^m  \sum_{y\notin U}|\xi_j(ty)|^2 = \frac{1}{m}\sum_{j=1}^m \mu_j(tU^c)= \mu(tU^c) < \eps.
\end{equation}

Denoting $\lambda_{G/G_x}(t^{-1})\xi_j$ by $v_j$ for each $j$, we obtain

$$ \langle \lambda_{G/G_x}(s)v_j, v_j \rangle = \sum_{y \in G\cdot x}  v_j(s^{-1}y)\overline{v_j}(y)
= \sum_{y \in U}  |v_j(y)|^2 + \sum_{y \notin U}v_j(s^{-1}y)\overline{v_j}(y).
$$

Now applying Cauchy-Schwarz and the fact that $U$ is $s$-invariant, we conclude
$$| 1- \langle \lambda_{G/G_x}(tst^{-1})\xi_j, \xi_j \rangle | \le  2 \sum_{y \notin U}  |v_j(y)|^2 = 2\mu_j(tU^c).$$
Averaging over $j$, and using $(\ref{muest})$ gives a contradiction to $(\ref{weak-cont})$ for $g=tst^{-1}$.
\end{proof}

\subsection{Proof of Theorem \ref{thm:c-star-simplicity-2} ((2) $\Rightarrow$ (1))} Since we require the notion of the reduced crossed product of a C*-algebra, we first briefly recall this construction.

Let $A$ be a $G$-C*-algebra, i.e. a unital C*-algebra endowed with an action of $G$ by automorphisms. Let $\pi: A \to B(\mathcal{H})$ be a faithful *-representation of $A$ into the space of bounded operators on a Hilbert space $\mathcal{H}$.

Let $\ell^2(G,\mathcal{H})$ be the Hilbert space of square summable $\mathcal{H}$-valued functions on $G$. The group $G$ acts unitarily on $\ell^2(G,\mathcal{H})$ by left translation,
\[
\lambda_{x} f (g) := f(x^{-1}g), \quad f \in \ell^2(G,\mathcal{H}),\ x,g \in G.
\]
We define a *-representation $\sigma : A \to B(\ell^2(G,\mathcal{H}))$ by
\[
\sigma(a)(f)(g) := \pi(g^{-1} \cdot a)f(g), \quad a \in A,\ f \in \ell^2(G,\mathcal{H}),\ g \in G.
\]

The reduced crossed product C*-algebra $A \rtimes_r G$ is the closure in $B(\ell^2(G,\mathcal{H}))$ of the subalgebra generated by the operators $\sigma(a)$ and $\lambda_x$. Note that $G$ acts isometrically on $A \rtimes_r G$ via conjugation:
\begin{equation}
\label{conj-action} \lambda_x \sigma(a) \lambda_x^* = \sigma(x \cdot a), \quad a \in A,\ x \in G.
\end{equation}
for all $x \in G$ and $a \in A$.

The isomorphism class of $A \rtimes_r G$ does not depend on the representation $\pi$ (see \cite{bo}*{Proposition 4.1.5}). Since the *-representation $\sigma$ is faithful, we will identify $A$ with its image under $\sigma$. For $a \in A$, we will denote $\sigma(a)$ by $a$.

We also recall that a completely positive map $\phi: A \to B$ between two C*-algebras $A$ and $B$ is a linear map such that each of the induced maps $\phi_n: M_n(A)\to M_n(B)$ for $n \in \bN$ is positive. Here, $M_n(A) \simeq M_n \otimes A$ denotes the C*-algebras of $n \times n$ matrices with coefficients in $A$, and $\phi_n \simeq \operatorname{id}_{M_n} \otimes \phi$. Clearly, every $*$-homomorphism between $A$ and $B$ is completely positive. Arveson's Extension Theorem (see \cite{bo}*{Theorem 1.6.1}) states that if $A \subset B$, then every completely positive map $A \to B(\mathcal{H})$ lifts to a completely positive map $B \to B(\mathcal{H})$.

Although we will not use it directly, we also recall Stinespring's fundamental structure theorem for completely positive maps (see \cite{bo}*{Theorem 1.5.3}), according to which every completely positive map $\phi : A \to B(\mathcal{H})$ is of the form $\phi(a)=V^* \pi(a) V$, where $V: \mathcal{H} \to \mathcal{K}$ is a linear map between Hilbert spaces, and $\pi: A \to B(\mathcal{K})$ is a *-representation of $A$.

The reduced crossed product $A \rtimes_r G$ is equipped with a $G$-equivariant unital completely positive map $E: A \rtimes_r G \to A$, the \emph{canonical conditional expectation}. For a finitely supported element $\sum_{x \in G} a_x \lambda_x \in A \rtimes_r G$,
\begin{equation}
\label{can-cond}E \left( \sum_{x \in G} a_x \lambda_x \right) = a_e.
\end{equation}

The commutative C*-algebra $C(X)$ of continuous functions on a compact $G$-space $X$ is a $G$-C*-algebra with respect to the action $(g\cdot f)(x)=f(g^{-1}x)$. When $X$ is a single point, then $C(X) \simeq \bC$, and the resulting reduced crossed product coincides with the reduced C*-algebra $\ca_r(G)$.

Finally, we recall the notion of \emph{multiplicative domain} $D_\phi$ of a completely positive map $\phi: A \to B$. Let
\begin{equation}\label{mult-dom}D_\phi:=\{ a \in A : \phi(a^*a)=\phi(a)^*\phi(a)\ \text{and}\ \phi(aa^*) = \phi(a)\phi(a)^*\}.\end{equation}
Then $\phi(ab)=\phi(a)\phi(b)$ and $\phi(ba)=\phi(b)\phi(a)$ for every $b \in A$ and $a \in D_\phi$, and so $D_\phi$ is a C*-subalgebra (see \cite{bo}*{Proposition 1.5.6}).

\begin{proof}[Proof of Theorem \ref{thm:c-star-simplicity-2} ((2) $\Rightarrow$ (1))] Let $\pi: \ca_r(G) \to B(\mathcal{H})$ be a non-trivial unital *-representation of the reduced C*-algebra of $G$. We need to show that $\pi$ is injective. We view $\ca_r(G)$ as a $G$-subalgebra of the reduced crossed product $\rcp{C(\fb)}{G}$. Since $\pi$ is completely positive we can apply Arveson's Extension Theorem to lift $\pi$ to a unital completely positive map $\phi:\rcp{C(\fb)}{G} \to B(\mathcal{H})$.

Notice that $\ca_r(G)$ belongs to the multiplicative domain (see $(\ref{mult-dom})$) of $\phi$, since $\phi$ extends $\pi$. Hence $\phi(\lambda_y a \lambda_x) = \pi(\lambda_x) \phi(a) \pi(\lambda_y)$ for all $x,y \in G$ and $a \in \rcp{C(\fb)}{G}$, and in particular $\phi$ is $G$-equivariant.

We are going to show that every $G$-equivariant unital completely positive map $\phi$ from $\rcp{C(\fb)}{G}$ to $B(\mathcal{H})$ is faithful, i.e. is nonzero on positive elements. This will clearly imply that $\pi$ is injective, as desired. Proposition \ref{rigidity} implies that the restriction $\phi_{|C(\fb)}$ of $\phi$ to the subalgebra $C(\fb)$ is an isometry onto its image. Applying Proposition \ref{injectivity}  again (along with Remark \ref{op-sys}) to the inverse $(\phi_{|C(\fb)})^{-1}$, we obtain a $G$-equivariant unital positive map $\tau: \operatorname{Im}(\phi) \to C(\fb)$. Then $\psi:=\tau \circ \phi$ is a $G$-equivariant unital completely positive map from $\rcp{C(\fb)}{G}$ to $C(\fb)$.

We claim that $\psi$ is the canonical conditional expectation, i.e. that $\psi(\lambda_s) = 0$ for every $s \in G \setminus \{e\}$. Since the canonical conditional expectation is faithful, this will imply that $\phi$ is faithful, and hence that $\pi$ is faithful. To see the claim, note that $\psi$ is the identity on $C(\fb)$ by Proposition \ref{rigidity}, and thus $C(\fb)$ belongs to the multiplicative domain of $\psi$. In particular, for every $s \in G$ and $f \in C(\fb)$ we have
\begin{equation}\label{trans}  \psi(\lambda_s)f = \psi(\lambda_s f) =\psi((s \cdot f) \lambda_s) =   (s\cdot f) \psi(\lambda_s).\end{equation}
But if the $G$-action on $\fb$ is free, then for $s \in G \setminus \{e\}$ and every $x \in \fb$, there is $f \in C(\fb)$ such that $f(x) \ne 0$ but $(s \cdot f)(x) = f(s^{-1}x)  = 0$. It follows that $\psi(\lambda_s)(x)=0$, and hence $\psi(\lambda_s)=0$, as desired.
\end{proof}

\begin{rem} In Section \ref{sec:cross-products} we will present yet another proof of the implication (2) $\Rightarrow$ (1) of Theorem \ref{thm:c-star-simplicity-2}, which is closer to the original argument in \cite{KK2014}. It runs as follows: by Archbold-Spielberg \cite{AS1994}*{Theorem 1}, if $X$ is a topologically free minimal compact $G$-space, then the reduced crossed product $C(X)\rtimes_r G$ is simple. We will show in Lemma \ref{lem:ideal-gen-nontriv-ideal} that if $X$ is a boundary, even non-topologically free, then this implies that $G$ is C*-simple.
\end{rem}

\subsection{Another characterization of C*-simplicity}
If $G$ is C*-simple, then for every amenable subgroup $H \leq G$, the quasi-regular representation $\lambda_{G/H}$ is weakly equivalent to $\lambda_G$. It turns out that this is actually a characterization of C*-simplicity.

\begin{thm}\label{newchar} A discrete group $G$ is C*-simple if and only if for every amenable subgroup $H \leq G$, the quasi-regular representation $\lambda_{G/H}$ is weakly equivalent to $\lambda_G$.
\end{thm}

\begin{proof} The ``only if'' direction is clear from the definition of C*-simplicity (see Remark \ref{weak-cont-def}), since $\lambda_{G/H}$ is weakly contained in $\lambda_G$ whenever $H$ is amenable. The converse follows from Proposition \ref{prop:not-weakly-contained-quasi-regular}, which we can apply to $X=\fb$ because of Theorem \ref{thm:c-star-simplicity-2}, Lemma \ref{topfree} and the amenability of point stabilizers (Proposition \ref{amen}).
\end{proof}

\section{Uniqueness of the trace and amenable IRS} \label{sec:uniqueness-of-trace}

We recall that a discrete group is said to have the {\em unique trace property} if its reduced C*-algebra has a unique tracial state, i.e. if the canonical trace $\tau_\lambda$ defined by
\[
\tau_\lambda(a) = \langle a \delta_e, \delta_e \rangle, \quad a \in \ca_r(G),
\]
is the only $G$-equivariant state on $\ca_r(G)$.

\begin{thm} \label{thm:unique-trace}
Let $G$ be a discrete group. Then every tracial state $\tau$ on the reduced C*-algebra $\ca_r(G)$ concentrates on the amenable radical $R_a(G)$. That is, $s \notin R_a(G)$ implies that $\tau(\lambda_s)=0$ for every $s \in G$.
\end{thm}

\begin{proof}The proof is similar to the proof of the implication $(2) \Rightarrow (1)$ of Theorem \ref{thm:c-star-simplicity-2}. Let $\tau : \ca_r(G) \to \bC$ be a $G$-equivariant state. Identify $\bC$ with the scalar subalgebra of $C(\fb)$. Then by Theorem \ref{injectivity}, we can extend $\tau$ to a $G$-equivariant unital positive  map $\psi : \rcp{C(\fb)}{G} \to C(\fb)$. Note that $\psi$ is actually completely positive (see \cite{paulsen}*{Theorem 3.9}).

Observe that $\psi(\lambda_s)=\tau(\lambda_s)$ is a constant function on $\fb$ for each $s \in G$. If $s \notin R_a(G)$, then Proposition \ref{furman} shows that it does not act trivially on $\fb$, and hence there is $x \in \fb$ such that $s\cdot x \neq x$. Choosing $f \in C(\fb)$ such that $f(x) \ne 0$ but $f(s^{-1}x)=0$ as before shows that $\psi(\lambda_s)=0$, as desired.
\end{proof}

\begin{rem}
More generally, if $A$ is any $G$-C*-algebra, then it follows as in the proof of Theorem \ref{thm:unique-trace} that every $G$-equivariant positive functional $\phi$ on $\rcp{A}{G}$ concentrates on $\rcp{A}{R_a(G)}$.
\end{rem}

\begin{cor}
A discrete group has the unique trace property if and only if its amenable radical is trivial. In particular, every C*-simple group has the unique trace property.
\end{cor}

\begin{proof}
The ``if'' direction follows immediately from Theorem \ref{thm:unique-trace}. Conversely, if $N$ is a non-trivial amenable normal subgroup of $G$, then any trace $\tau_0$ on the reduced C*-algebra $\ca_r(N)$ gives rise to a trace $\tau$ on the reduced C*-algebra $\ca_r(G)$ via $\tau = \tau_0 \circ E_{N}$, where $E_{N}$ denotes the canonical conditional expectation from $\ca_r(G)$ onto $\ca_r(N)$ that satisfies $E_N(\lambda_g) = 0$ for $g \in G \setminus N$. Since $N$ is amenable, $\lambda_N$ weakly contains the trivial representation, and hence the unit character on $N$ that sends every element to $1$ extends to a non-canonical trace $\tau_0$ on $\ca_r(N)$.
\end{proof}

Let $G$ be a discrete group, and let $\S(G)$ denote the set of subgroups of $G$. The set $\S(G)$ is compact with respect to the Chabauty topology, which corresponds to the product topology on $\{0,1\}^G$, and $\S(G)$ forms a $G$-space with respect to the conjugation action of $G$.

An {\em invariant random subgroup} of $G$ is a probability measure $\mu$ on $\S(G)$ that is invariant with respect to the adjoint of the conjugation action of $G$ on $\S(G)$. Let $\S_a(G)$ denote the set of amenable subgroups of $G$. Then $\mu$ is said to be {\em amenable} if $\S_a(G)$ is $\mu$-measurable and $\mu(\S_a(G)) = 1$.

The notion of an invariant random subgroup was introduced in \cite{AGV2014}, and the problem was raised whether every amenable invariant random subgroup is concentrated on the amenable radical. This problem was recently solved affirmatively in \cite{BDL2014}. Combining Theorem \ref{thm:unique-trace} and \cite{T2012}*{Corollary 5.15}, we obtain a different proof.

\begin{cor}
Every amenable invariant random subgroup on a discrete group is concentrated on the amenable radical.
\end{cor}

\section{Stability under group extensions}\label{stability}

In this section, we establish Theorem \ref{thm:c-star-simplicity-closed-wrt-extensions} from the introduction. In particular, this yields the stability of C*-simplicity under group extensions. This will require some preliminary lemmas. For a group $G$ and a subgroup $H \leq G$, the centralizer of $H$ is denoted by $C_G(H) = \{s \in G \mid st = ts\ \forall t \in H\}$.

\begin{lem} \label{lem:elements-not-fixing-generate-group}
Let $N$ be a discrete group, let $X$ be an $N$-boundary, and let $U \subset X$ be a non-empty open subset. Then the set $\{t \in N \mid tU \cap U \ne \emptyset \}$ generates $N$.
\end{lem}

\begin{proof}
Let $H \le N$ denote the subgroup generated by $\{t \in N \mid tU \cap U \ne \emptyset \}$. Then $HU$ is a non-empty open subset of $X$ such that $tHU \cap HU = \emptyset$ for all $t \in N \setminus H$. By minimality and compactness, $(tHU)_{tH \in N/H}$ is necessarily a finite partition of $X$. Note in particular that $N/H$ is finite. The corresponding equivalence relation on $X$ induces a continuous equivariant map from $X$ to $N/H$. Since $X$ is proximal, $N/H$ is proximal. Being finite, it follows that $N/H$ is a singleton, and hence that $H = N$.
\end{proof}

\begin{lem} \label{lem:boundary-action-induced-by-normal-subgroup}
Let $G$ be a discrete group and let $N \le G$ be a normal subgroup with universal $N$-boundary $\partial_F N$. The $N$-action on $\partial_F N$ extends to a $G$-boundary action on $\partial_F N$.
\end{lem}

\begin{proof}
For $s \in G$, define $\sigma_s \in \aut(N)$ by $\sigma_s(t) = sts^{-1}$ for $t \in N$. Mapping $N$ into $\aut(N)$ via $\sigma$, it follows from \cite{G1976}*{Proposition II.4.3} that the $N$-action on $\partial_F N$ has an extension to an action of $\aut(N)$ on $\partial_F N$. Composing this action with the map from $G$ into $\aut(N)$ gives a $G$-action on $\partial_F N$ that extends the $N$-action. Since $\partial_F N$ is an $N$-boundary, it is clear that this is a $G$-boundary action
\end{proof}

\begin{lem} \label{lem:acts-freely-or-trivially-centralizer}
Let $G$ be a discrete group and let $N \le G$ be a C*-simple normal subgroup with universal $N$-boundary $\partial_F N$. Then the action of $s \in G$ on $\partial_F N$ is either trivial or free, where the $G$-action on $\partial_F N$ is defined as in Lemma \ref{lem:boundary-action-induced-by-normal-subgroup}. The former possibility occurs if and only if $s \in C_G(N)$.
\end{lem}

\begin{proof}
If $s \in G$ belongs to $C_G(N)$, then the automorphism $\sigma_s \in \aut(N)$ defined as in the proof of Lemma \ref{lem:boundary-action-induced-by-normal-subgroup} is trivial, and it follows from the construction of the $G$-action on $\partial_F N$ that $s$ acts trivially on $\partial_F N$.

For the converse, fix $s \in G$ such that the set of $s$-fixed points $Fix(s)$ is non-empty. For every $t \in N$ such that $tFix(s) \cap Fix(s) \ne \emptyset$, the actions of the elements $sts^{-1}$ and $t$ coincide on $Fix(s) \cap t^{-1}Fix(s)$. However, since $N$ is C*-simple, Theorem \ref{thm:c-star-simplicity-2} implies that $N$ acts freely on $\partial_F N$. Thus $sts^{-1} = t$. By Lemma \ref{lem:elements-not-fixing-generate-group}, the elements with this property generate $N$. Hence $s \in C_G(N)$.
\end{proof}

We are now ready for the proof of Theorem \ref{thm:c-star-simplicity-closed-wrt-extensions}.

\begin{proof}[Proof of Theorem \ref{thm:c-star-simplicity-closed-wrt-extensions}]
For brevity, let $K = C_G(N)$ and let $L = NK$. Note that $K$ and $L$ are normal in $G$.

Consider the following three $G$-boundary actions. First, by Lemma \ref{lem:boundary-action-induced-by-normal-subgroup}, the $N$-action on $\partial_F N$ extends to a $G$-boundary action on $\partial_F N$. Similarly, the $K$-action on $\partial_F K$ extends to a $G$-action on $\partial_F K$. Finally, we have a $G$-boundary action on $\partial_F(G/L)$.

By the construction of these $G$-actions, $N$ acts trivially on $\partial_F K$ and $\partial_F(G/L)$, but minimally on $\partial_F N$. On the other hand, $K$ acts trivially on $\partial_F N$ and $\partial_F(G/L)$, but minimally on $\partial_F K$. Since $G$ acts minimally on $\partial_F(G/L)$, it is not difficult to see that the diagonal $G$-action on
\[
X := \partial_F N \times \partial_F K \times \partial_F(G/L)
\]
is a boundary action.

First suppose that both $N$ and $K$ are C*-simple. Then by Theorem \ref{thm:c-star-simplicity-2}, the $N$-action on $\partial_F N$ and the $K$-action on $\partial_F K$ are free. If $s \in G$ does not act freely on $X$, then in particular $s$ does not act freely on either $\partial_F N$ or $\partial_F K$. Hence by Lemma \ref{lem:acts-freely-or-trivially-centralizer}, $s \in K$. Then since $K$ acts freely on $\partial_F K$, it follows that $s = e$. Therefore, $G$ acts freely on $X$, and hence $G$ is C*-simple by Theorem \ref{thm:c-star-simplicity-2}.

Conversely, suppose that $G$ is C*-simple. For $(x,y,z) \in X$, consider the stabilizer $G_{(x,y,z)}$. Lemma \ref{amen} implies that $G_{(x,y,z)} \cap L = N_x K_y$ and $G_{(x,y,z)}/(G_{(x,y,z)} \cap L) \subset (G/L)_z$ are both amenable. Hence $G_{(x,y,z)}$ is amenable. Therefore, by Proposition \ref{prop:not-weakly-contained-quasi-regular}, the $G$-action on $X$ is topologically free. It follows that the $N$-action on $\partial_F N$ and the $K$-action on $\partial_F K$ are also topologically free. Hence by Theorem \ref{thm:c-star-simplicity-2}, $N$ and $K$ are C*-simple.

Finally, to see that C*-simplicity is closed under extension, suppose that $N$ and $G/N$ are C*-simple. Observe that $N$ necessarily has trivial center, so $C_G(N)$ is isomorphic to a normal subgroup of $G/N$. The above results imply that $C_G(N)$ is C*-simple, and hence that $G$ is C*-simple.
\end{proof}

\begin{rem}
Note that Theorem \ref{thm:c-star-simplicity-closed-wrt-extensions} provides a negative answer to \cite{DP2011}*{Question (Q)}.
Also note that the analogous result for the property of having trivial amenable radical is proved in \cite{T2012}*{Lemma B.6}.
\end{rem}

\begin{rem}[stability under finite index subgroups]
Theorem \ref{thm:c-star-simplicity-closed-wrt-extensions} also straightforwardly implies that if $H$ has finite index in $G$, then C*-simplicity of $G$ implies that of $H$, while vice versa the C*-simplicity of $H$ and the absence of non trivial finite normal subgroup in $G$ implies the C*-simplicity of $G$, (see also \cite{D2007}*{Proposition 19}).
\end{rem}

\section{Examples of C*-simple groups}\label{sec:examples}
In this section, we put forward a simple criterion, which implies C*-simplicity, and can be used to easily prove the C*-simplicity of many groups.

\begin{defn}[Normalish subgroups]
Let $G$ be a group. A subgroup $H \le G$ is said to be  {\em normalish} if for every $n \geq 1$ and $t_1,\ldots,t_n \in G$ the intersection $\cap_i t_i H t_i^{-1}$ is infinite.
\end{defn}

This notion is slightly stronger than the notion of an $n$-step s-normal subgroup, which was introduced in \cite{BFS2013}. Note that a subgroup that is $n$-step s-normal for every $n \geq 1$ is normalish. Our main observation is the following.

\begin{thm}\label{thm:not-c-star-simple-implies-normalish} A discrete group $G$ with no non-trivial finite normal subgroups and no amenable normalish subgroups is C*-simple.
\end{thm}

\begin{proof} Let $G$ be a discrete group with no non-trivial finite normal subgroup. To show that $G$ is C*-simple, it is enough  to prove that $G$ acts topologically freely on $\fb$, according to Theorem \ref{thm:c-star-simplicity-2}. If that is not the case, we claim that each $G_x$, $x \in \fb$ is normalish. Indeed let $s \in G\setminus \{e\}$ have a fixed point set $Fix(s)$ with non empty interior. By strong proximality any finite set of points $x_1,\ldots,x_k$ in $\fb$ can be mapped into the interior of $Fix(s)$ by some $g \in G$, so that $g^{-1}sg$ fixes each $x_i$. In particular this applies to $x_i=t_i x$, showing that the intersection $\cap_i t_i G_x t_i^{-1}$ is non trivial. If this is finite for some choice of $t_i$'s, but non trivial for all choices of $t_i$'s, then $\cap_{g \in G} gG_x g^{-1}$ is a non trivial finite group normal subgroup, contrary to our assumption. Hence $G_x$ is normalish for each $x \in \fb$. Finally since point stabilizers are also amenable by Lemma \ref{amen}, this ends the proof.
\end{proof}

We note however that the converse does not hold (see Subsection \ref{sec:baumslag-solitar}). In the next two subsections we apply this criterion to show the C*-simplicity of many new groups and recover that  of virtually all previously known examples.

\subsection{Bounded cohomology and $\ell^2$-Betti numbers} \label{sec:cohomology}

We recall that a group is said to be C*-simple if its reduced C*-algebra is simple. In this section we will prove the C*-simplicity of discrete groups with trivial amenable radical, supposing that they have either non-trivial bounded cohomology, or non-vanishing $\ell^2$-Betti numbers. The key idea is that these conditions preclude the existence of amenable normalish subgroups.

Let $G$ be a discrete group. For $n \geq 0$, let $\beta_n^{(2)}$ denote the $n$-th $\ell^2$-Betti number of $G$ (see e.g. \cite{L1998}). For the result about non-vanishing $\ell^2$-Betti numbers, we require the following special case of \cite{BFS2013}*{Theorem 1.3}.

\begin{prop}\label{betti}
Let $G$ be a discrete group. If $G$ contains an amenable normalish subgroup, then $\beta_n^{(2)} = 0$ for every $n \geq 0$.
\end{prop}

Next, we consider the case when $G$ has non-trivial bounded cohomology. Recall (see e.g. \cite{M2001}) that a {\em coefficient $G$-module} $(\pi,E)$ is an isometric linear $G$-representation $\pi$ on a dual Banach space $E$ (specifically, the dual of a separable Banach space) such that the operators in the image of $\pi$ are weak*-continuous.

For $n \geq 0$, the {\em bounded cohomology} group $H_b^n(G,E)$ of $G$ with coefficient module $(\pi,E)$ is the $n$-th cohomology group of the homogeneous cochain complex
\[
0 \longrightarrow \ell^\infty(G,E)^G \stackrel{d_1}{\longrightarrow}  \ell^\infty(G^2,E)^G \stackrel{d_2}{\longrightarrow} \ell^\infty(G^3,E)^G \stackrel{d_3}{\longrightarrow} \cdots
\]
consisting of bounded $G$-invariant functions. The $G$-action on $\ell^\infty(G^n,E)^G$ is given by
$(s \cdot f)(x_0,\ldots,x_n) = f(s^{-1}x_0,\ldots,s^{-1}x_n)$ and
\[
(d_n f)(s_0,\ldots,s_{n})=\sum_{j=0}^{n}(-1)^j f(s_0,\ldots,s_{j-1},s_{j+1},\ldots,s_{n})
\]
for $s \in G$ and $f \in \ell^\infty(G^n,E)^G$. Namely, $H_{\mathrm{b}}^n(G,E)=\ker d_{n+1} / \ran d_n$. See \cite{M2001} for details.

The dual Banach $G$-module $E$ is said to be {\em mixing} if the stabilizer subgroup $G_x = \{ s \in G \mid sx = x \}$ is finite for every $x \in E \setminus \{0\}$. Examples of such $G$-modules include $\ell^p(G)$ for $1\le p < \infty$.

The proof of the next result is similar to the proofs of \cite{M2001}*{Corollary 7.5.9, Corollary 7.5.10}.

\begin{prop} \label{thm:trivial-bdd-cohomology}
Let $G$ be a discrete group, and let $(\pi,E)$ be a coefficient $G$-module with $E$ mixing. If $G$ contains an amenable normalish subgroup, then $H_b^n(G,E)$ is trivial for every $n \geq 0$ and every mixing dual Banach $G$-module.
\end{prop}

\begin{proof}
Let $H \le G$ be an amenable normalish subgroup. We can compute the bounded cohomology of $G$ using the  complex
\[
0 \longrightarrow \ell^\infty(G/H,E)^G \longrightarrow  \ell^\infty((G/H)^2,E)^G \longrightarrow \ell^\infty((G/H)^3,E)^G \longrightarrow \cdots
\]
We claim that $\ell^\infty((G/H)^n,E)^G$ vanishes, and hence that $H_b^n(G,E)$ also vanishes. To see this, fix $f \in \ell^\infty((G/H)^n,E)^G$. Then for every $t = (t_1,\ldots,t_n) \in (G/H)^n$, the element $f(t) \in E$ is left invariant by every element in $\cap_i t_i H t_i^{-1}$. Since $E$ is mixing, it follows that $f(t) = 0$.
\end{proof}

We will say that a discrete group $G$ has {\em non-trivial bounded cohomology} if there is a coefficient $G$-module $(\pi,E)$ with $E$ mixing such that $H_b^n(G,E)$ is non-trivial for some $n \geq 0$.

\begin{thm} \label{thm:vanishing-cohomology-or-betti-numbers-c-star-simple}
Let $G$ be a discrete group with trivial amenable radical such that either
\begin{enumerate}
\item $G$ has non-trivial bounded cohomology, or
\item $G$ has non-vanishing $\ell^2$-Betti numbers.
\end{enumerate}
Then $G$ is C*-simple.
\end{thm}

\begin{proof} Proposition \ref{betti} and Proposition \ref{thm:trivial-bdd-cohomology} imply that $G$ has no amenable normalish subgroup. The conclusion follows from Theorem \ref{thm:not-c-star-simple-implies-normalish}.
\end{proof}

The class of groups $\C_{\text{reg}}$ was introduced in \cite{MS2006}*{Notation 1.2}. It consists of those countable discrete groups $G$ satisfying $H_b^2(G,\ell^2(G)) \ne 0$, which can be seen as a cohomological analogue of the property of having negative curvature. This includes groups admitting a non-elementary proper isometric action on some Gromov-hyperbolic graph of bounded valency, groups admitting a non-elementary proper isometric action on some proper CAT(-1) space, and groups admitting a non-elementary simplicial action on some simplicial tree.

The closely related class of groups $\D_{\text{reg}}$ was introduced in \cite{T2009}*{Definition 2.6} as a variation on the class $\C_{\text{reg}}$. It consists of those countable discrete groups with the property that there exists an unbounded quasi-cocycle from $G$ to $\ell^2(G)$.

The classes $\C_{\text{reg}}$ and $\D_{\text{reg}}$ both properly contain the class of acylindrically hyperbolic groups introduced in \cite{O2013}. This latter class includes all non-elementary hyperbolic and relatively hyperbolic groups, outer automorphism groups of free groups on two or more generators, all but finitely many mapping class groups of punctured closed surfaces and most 3-manifold groups. It was proved in \cite{DGO2011}*{Theorem 2.32} that an acylindrically hyperbolic group is C*-simple if and only if its amenable radical is trivial.

By \cite{T2009}*{Lemma 2.8}, every group in $\D_{\text{reg}}$ has either non-vanishing first $\ell^2$-Betti number or non-trivial second bounded cohomology with coefficients in $\ell^2(G)$. Therefore, Theorem \ref{thm:vanishing-cohomology-or-betti-numbers-c-star-simple} implies the following generalization of \cite{DGO2011}*{Theorem 2.32}.

\begin{cor}
A group in $\C_{\text{reg}}$ or $\D_{\text{reg}}$ is C*-simple if and only if its amenable radical is trivial.
\end{cor}

Recall that a discrete group is said to be {\em strongly non-amenable} if it has positive first $\ell^2$-Betti number.

\begin{cor}
A strongly non-amenable group is C*-simple if and only if its amenable radical is trivial.
\end{cor}

We note that by \cite{G2000}*{Lemme V.3}, if a group $G$ contains an amenable $1$-step s-normal subgroup, then the cost is zero for every probability measure preserving action of $G$.

\subsection{Linear groups}\label{sec:linear}

A linear group is a subgroup of $GL_n(K)$ for some field $K$. A discrete linear group is a linear group endowed with the discrete topology.

\begin{thm}\label{thm:linear-groups-non-amenability-normalish}
Let $G$ be a discrete linear group. If $H \le G$ is an amenable subgroup, then there is a finite subset $F \subset G$ such that $\bigcap_{t \in F} tHt^{-1}$ is contained in the amenable radical of $G$. In particular, if $G$ has trivial amenable radical, then every normalish subgroup of $G$ is non-amenable.
\end{thm}

\begin{proof}
Let $G$ be a subgroup of $GL_n(K)$ for some algebraically closed field $K$ and let $H \le G$ be an amenable subgroup. We may assume that $R_a(G) \le H$.

For a finite subset $F \subset G$, let $H_F = \cap_{t \in F} tHt^{-1}$, and let $L_F$ denote the Zariski closure of $H_F$, which is an algebraic subgroup of $GL_n(K)$. Applying the descending chain condition for varieties, the intersection $L = \cap_F L_F$ over finite subsets $F \subset G$ is actually a finite intersection, and hence $L = L_{F_0}$ for some finite subset $F_0 \subset G$.

Observe that for $t \in G$, $t L_{F_0} t^{-1} = L_{tF_0}$. But by construction, $L_{F_0} \subset  L_{tF_0}$. Thus, $L \subset t L t^{-1}$, and it follows that $L$ is normalized by $G$. Hence $L \cap G$ is a normal subgroup of $G$. We will show that $L \cap G$ is amenable. Since $H_{F_0} \subset L \cap G$, this will prove the result.

Suppose for the sake of contradiction that $L \cap G$ is not amenable. Let $A = H_{F_0}$. Then $A$ is an amenable subgroup of $G$ with the property that $A_F$ is Zariski dense for every finite subset $F \subset G$. Hence by replacing $G$ with $L \cap G$, we may assume that $G$ is a subgroup of the algebraic group $L$ containing an amenable subgroup $A$ such that $A_F$ is Zariski-dense in $L$ for every finite subset $F \subset G$.

Replacing $G$ with $L^0 \cap G$, where $L^0$ is the connected component of the identity in $L$, we may further assume that $L$ is connected as an algebraic group. Observe that $L$ is not solvable, for otherwise $G$ would be amenable. Hence $L$ admits a center-free simple quotient. Projecting $G$ to that quotient, we see that we may assume additionally that $L$ is a connected center-free simple algebraic $K$-group.

Recall that according to the Tits alternative \cite{Ti1972} every amenable finitely generated linear group is virtually solvable. In characteristic zero this remains true without the finite generation assumption. So first suppose that $K$ has characteristic zero. Then we may conclude that $A$ is virtually solvable. Hence $L$ is also virtually solvable, contradicting our assumption that $L$ is simple.

Now suppose that $K$ has characteristic $p > 0$. We first claim that $A$ is locally finite. To see this, we require the following definition. For an arbitrary subgroup $\Gamma$ of $GL_n(K)$, let $\core(\Gamma)$ denote the inductive limit
\[
\core(\Gamma) = \underrightarrow{\lim} \operatorname{cls}(\Lambda)^0,
\]
as $\Lambda$ ranges over all finitely generated subgroups of $\Gamma$, where $\operatorname{cls}(\Lambda)^0$ denotes the connected component of the identify of the Zariski closure of $\Lambda$. Then $\core(\Gamma)$ is a connected algebraic subgroup of $GL_n(K)$ that is normalized by $\Gamma$ and $\core(\Gamma) \cap \Gamma$ is Zariski-dense in $\core(\Gamma)$. Moreover, $\Gamma$ is locally finite if and only if $\core(\Gamma)$ is trivial.

Now $\core(A)$ is a connected algebraic subgroup that is normalized by $A$, and hence by $L$. According to the Tits alternative every finitely generated subgroup of $A$ is virtually solvable. Hence its Zariski closure has a solvable connected component of the identity. It follows that $\core(A)$ is solvable. Since $L$ is simple, we conclude that $\core(A)$ is trivial, and hence that $A$ is locally finite.

Let $\overline{\mathbb{F}_p}$ denote the algebraic closure of the prime field $\mathbb{F}_p$ in $K$. Let $GL_n(K)$ act on $M_n(K)$ by left multiplication. We claim that there is a $G$-invariant $K$-vector subspace $W$ in $M_n(K)$, which admits a $\overline{\mathbb{F}_p}$-structure, i.e. an $\overline{\mathbb{F}_p}$-vector subspace $W_0$ such that $W=W_0 \otimes_{\overline{\mathbb{F}_p}} K$, such that $A$ preserves $W_0$. Since $W$ is $G$-invariant, it is also $L$-invariant, and making $L$ act on $W$, we obtain an embedding of $L$ as an algebraic subgroup of $GL_m(K)$, $m=\dim W$. Under this embedding $G \le GL_m(K)$ and $A \le GL_m(\overline{\mathbb{F}_p})$ are Zariski-dense in $L$. Hence, up to changing $n$ into $m$, and modulo the claim, we may assume that $A \le GL_n(\overline{\mathbb{F}_p})$.

Now it is easy to reach a contradiction: for every $t \in G$, $A \cap t A t^{-1}$ is Zariski-dense in $L$. In particular, since $L$ is center-free, the intersection of the centralizers $C_L(h):=\{x \in L; xa=ax\}$ for $a$ varying in $A \cap t A t^{-1}$ is trivial. And by the descending chain condition, there is a finite set of $a$'s, say $a_1,\ldots,a_k$ in $A \cap t A t^{-1}$ such that $\cap_1^k C_L(a_i)$ is trivial. By assumption $t^{-1}a_it \in GL_n(\overline{\mathbb{F}_p})$ for each $i=1,\ldots,k$. Thus there is a Galois automorphism $\sigma$ of $K$, a power of the Frobenius map $x \mapsto x^p$, such that $\sigma(t^{-1}a_it) = t^{-1}a_it$ for each $i=1,\ldots,k$. This implies that $\sigma(t)t^{-1}$ commutes with each $a_i$, and hence is trivial. This means that $\sigma(t)=t$, i.e. that $t \in GL_n(\overline{\mathbb{F}_p})$. Hence  $ G \le GL_n(\overline{\mathbb{F}_p})$, and $G$ is locally finite, hence amenable, which is a contradiction.

It remains to verify the claim. Note that since every element $a$ of $A$ has finite order, its eigenvalues belong to $\overline{\mathbb{F}_p}$. Now consider the trace $tr(xy)$ on $M_n(K)$ as a non-degenerate bilinear form. Let $W$ be the $K$-vector subspace of $M_n(K)$ generated by all matrices $a \in A$. Pick a basis $a_1,\ldots,a_m \in A$ of $W$. Then the linear map sending $w \in W$ to the $m$-tuple $(tr(wa_1),\ldots,tr(wa_m))$ is a $K$-linear isomorphism, which sends the $\overline{\mathbb{F}_p}$-linear span of the $a_i$'s to $(\overline{\mathbb{F}_p})^m$. This gives the desired $\overline{\mathbb{F}_p}$-structure on $W$. Since $A$ is Zariski-dense in $G$ and $L$, it follows that $W$ is fixed by $L$, proving the claim. This ends the proof.
\end{proof}

In an important paper \cite{BCH1994}, Bekka, Cowling and de la Harpe proved that lattices in semi-simple real Lie groups with trivial center are C*-simple. This was vastly extended by Poznansky \cite{P2008} to all linear groups in the form below. Combining Theorem \ref{thm:linear-groups-non-amenability-normalish} with Theorem \ref{thm:not-c-star-simple-implies-normalish}, we thus obtain a new proof of this result.

\begin{cor} \label{thm:linear-groups}
A discrete linear group is C*-simple if and only if its amenable radical is trivial.
\end{cor}

\begin{rem} Another route for proving this fact is to use Theorem \ref{thm:c-star-simplicity-2} by exhibiting a topologically free boundary action. As in the proof of the Tits alternative \cite{Ti1972}, or as in \cite{BCH1994}, passing to a finite index subgroup if necessary, and looking at various linear representations, one can let the linear group $G$ act proximally and strongly irreducibly on projective spaces over local fields. The associated limit sets (i.e. the closure of the top eigendirections of all proximal elements) are $G$-boundaries and they are topologically free modulo the kernel of the projective representation.  One can then proceed in stages using the stability under group extensions (Theorem \ref{stability}) and induction on the dimension of the Zariski closure of $G$.
\end{rem}

\subsection{Amalgamated free products and Baumslag-Solitar groups} \label{sec:baumslag-solitar}
In \cite{DP2011}*{Theorem 2} de la Harpe and Pr\'eaux consider the C*-simplicity of certain amalgamated free products and HNN extensions including the Baumslag-Solitar groups $BS(n,m) = \langle a,t \mid t^{-1} a^m t = a^n \rangle$ with $|n| \neq |m|$ and $|n|,|m| \ge 2$. Their analysis implies that the action of these groups on the boundary of their Bass-Serre tree is a topologically free boundary action. Thus C*-simplicity follows from Theorem \ref{thm:c-star-simplicity-2}.
We note that the cyclic subgroup $\langle a \rangle \le BS(n,m)$ is amenable and normalish.

\subsection{Groups with few amenable subgroups}
In this subsection, we present another criterion, which implies C*-simplicity and applies to certain new family of groups.

\begin{thm} \label{thm:few-amenable-subgroups}
A discrete group with only countably many amenable subgroups is C*-simple if and only if its amenable radical is trivial.
\end{thm}

\begin{proof}
Suppose that the amenable radical of $G$ is trivial, but that $G$ is not C*-simple. Then by Theorem \ref{thm:c-star-simplicity-2}, the $G$-action on $\fb$ is faithful (see Proposition \ref{furman}) but not topologically free. Let $s \in G$ be such that the set of $s$-fixed points $Fix(s)$ has non-empty interior. We claim that there is $y \in \operatorname{Fix}(s)$ such that stabilizer $G_y$ fixes a non-empty open subset of $\fb$ pointwise.

For $x \in \fb$, let $F_x = \cap_{t \in G_x} \operatorname{Fix}(t)$ denote the set of points in $\fb$ fixed by every element in the stabilizer $G_x$. Note that $F_x$ is closed and contains $x$. By Lemma \ref{amen}, the stabilizer $G_x$ is amenable. Therefore, by the assumption that $G$ contains only countably many amenable subgroups, there is a countable sequence $(x_k) \in \fb$ such that  $\cup_{x \in \operatorname{Fix}(s)} F_x = \cup_k F_{x_k}$. Since $\cup_{x \in \operatorname{Fix}(s)} F_x$ contains $\operatorname{Fix}(s)$ and thus has non empty interior, the Baire category theorem implies that there is $y \in \operatorname{Fix}(s)$ such that $F_y$ has non-empty interior, say $U$. This means that $G_y$ fixes $U$ pointwise and proves the claim.

On the other hand, by compactness and minimality, there are $s_1,\ldots,s_n$  in $G$ such that $s_1U,\ldots,s_nU$ cover $\fb$. Notice that the elements in $s_k G_y s_k^{-1}$ fix every point in $s_k U$. Therefore, the elements in the intersection $\cap_k s_k G_y s_k^{-1}$ fix every point in $\fb$. But $G_y$ is normalish (see the proof of Theorem \ref{thm:not-c-star-simple-implies-normalish}). So this contradicts the faithfulness of the action and ends the proof.
\end{proof}

For a prime number $p$, a Tarski monster group of order $p$ is an infinite group with the property that every non-trivial subgroup is cyclic of order $p$. Tarski monster groups of order $p$ were first constructed in \cite{O1980} for $p > 10^{75}$ as a counterexample to the Day-von Neumann conjecture about the amenability of groups which do not contain free subgroups. In \cite{O1991}, torsion-free Tarski monster groups were constructed. These are infinite groups with the property that every non-trivial subgroup is cyclic of infinite order.

It was shown in \cite{KK2014}*{Corollary 6.6}, using an ad hoc argument, that Tarski monster groups are C*-simple. However, since both Tarski monster groups and torsion-free Tarski monster groups are finitely generated and have trivial amenable radicals, they satisfy the hypotheses of Theorem \ref{thm:few-amenable-subgroups}. Thus we obtain the following generalization of \cite{KK2014}*{Corollary 6.6}.

\begin{cor}
Tarski monster groups and torsion-free Tarski monster groups are C*-simple.
\end{cor}

The free Burnside group of rank $m$ and exponent $n$, written $B(m,n)$ is an infinite group that is, in a certain specific sense, the ``largest'' group with $m$ generators such that every element in the group has order $n$. It was recently shown in \cite{OO2014} that $B(m,n)$ is C*-simple for $m \geq 2$ and $n$ odd and sufficiently large.

It was shown in \cite{I1994} that for $n$ sufficiently large, every non-cyclic subgroup of the free Burnside group $B(m,n)$ contains a subgroup isomorphic to $B(2,n)$. By \cite{A1983}, if $n$ is odd, then $B(2,n)$ is non-amenable. Thus for $m \geq 2$ and $n$ odd and sufficiently large, $B(m,n)$ satisfies the hypotheses of Theorem \ref{thm:few-amenable-subgroups}, and we recover the following result from \cite{OO2014}.

\begin{cor}
For $m \geq 2$ and $n$ odd and sufficiently large, the free Burnside group $B(m,n)$ is C*-simple.
\end{cor}

\section{Reduced crossed products and  point stabilizers}\label{sec:cross-products}

In this section, we further investigate the simplicity of reduced crossed product C*-algebras over C*-simple groups, and collect further information about the point stabilizers of boundary actions.

\subsection{Simplicity of reduced crossed products}\label{sec:cross}

Let $A$ be a {\em $G$-C*-algebra}, i.e. a C*-algebra equipped with a $G$-action. The reduced crossed product $\rcp{A}{G}$ is then also a $G$-C*-algebra with respect to the action of $G$ by conjugation $(\ref{conj-action})$. Recall the canonical conditional expectation $E : \rcp{A}{G} \to A$ defined in $(\ref{can-cond})$.

The following theorem answers a question of de la Harpe and Skandalis and generalizes \cite{DS1986}*{Theorem I}.

\begin{thm} \label{thm:simplicity-crossed-products}
Let $G$ be a discrete C*-simple group. Then for any unital $G$-C*-algebra $A$ having no non-trivial $G$-invariant closed ideal, the reduced crossed product $\rcp{A}{G}$ is simple.
\end{thm}

The proof of Theorem \ref{thm:simplicity-crossed-products} is divided into several steps.

\begin{lem} \label{lem:ideal-gen-nontriv-ideal}
Let $A$ be a unital $G$-C*-algebra and let $X$ be a $G$-boundary. Then for any non-trivial closed ideal $I$ of $\rcp{A}{G}$, the ideal $J$ of $\rcp{(A \otimes C(X))}{G}$ generated by $I$ is non-trivial.
\end{lem}

\begin{proof}
We proceed as in the proof of $(2) \Rightarrow (1)$ in \cite{KK2014}*{Theorem 6.2}.  Let $\pi : \rcp{A}{G} \to \IB(\cH)$ be a *-representation such that $\ker\pi=I$. We extend $\pi$ to a unital completely positive map $\bar{\pi}$ from $\rcp{(A\otimes C(X))}{G}$ into $\B(\cH)$.

Note that $A\rtimes_{\mathrm{r}} G$ is contained in the multiplicative domain of $\bar{\pi}$, that is $\bar{\pi}(afb)=\pi(a)\bar{\pi}(f)\pi(b)$ for every $a,b\in \rcp{A}{G}$ and $f\in C(X)$ (see \cite{bo}*{Proposition 1.5.6}). In particular, $\pi(A)$ and $\bar{\pi}(C(X))$ commute.

 If $\bar{\pi}$ were multiplicative on $C(X)$ we would be done, because then $\bar{\pi}$ would be a *-homomorphism and its kernel would contain $I$. This may not be the case, but since $X$ is a $G$-boundary, it is easy to check that $\bar{\pi}$ is completely isometric on $C(X)$. Since $\bar{\pi}$ extends to a completely isometric map on the injective envelope of $C(X)$, it follows as in the proof of  \cite{CE1977}*{Thm 3.1} that there is a *-homomorphism $Q$ from the C*-subalgebra $\ca(\bar{\pi}(C(X)))$  generated by $\bar{\pi}(C(X))$ onto $C(X)$ such that $Q\circ\bar{\pi}=\operatorname{id}_{C(X)}$. The C*-subalgebra $\ca(\bar{\pi}(C(X)))$ is a $G$-C*-algebra with the conjugation $G$-action through $\pi$, and the ideal $K=\ker Q$ is $G$-invariant.

We consider
\[
D=\ca\Bigl(\bar{\pi}\bigl(\rcp{(A\otimes C(X))}{G}\bigr)\Bigr)
=\overline{\ca(\bar{\pi}(C(X)))\cdot\pi(\rcp{A}{G})}
\]
and the ideal $L$ of $D$ defined by
\[
L= \overline{K\cdot \pi(A\rtimes_{\mathrm{r}} G)}.
\]
An element $d\in D$ belongs to $L$ if and only if $e_id \to d$ for an approximate unit $(e_i)$ of $K$. This implies that $L\cap \ca(\bar{\pi}(C(X))) = K$. Let $Q$ now denote the quotient map from $D$ onto $D/L$. Then $Q\circ\bar{\pi}$ is a *-homomorphism, since it is a unital completely positive map which is multiplicative on $A\otimes C(X)$ and $G$-equivariant. The ideal $\ker(Q\circ\bar{\pi})$ is proper and contains $I$.
\end{proof}

Let $B$ be a $G$-C*-algebra and let $K$ be a $G$-invariant closed ideal of $B$. Then, letting $K\mathbin{\bar{\rtimes}}_{\mathrm{r}} G$ denote the kernel of the map $B\rtimes_{\mathrm{r}} G \to (B/K)\rtimes_{\mathrm{r}} G$,
\[
K\mathbin{\bar{\rtimes}}_{\mathrm{r}} G = \{ b \in B\rtimes_{\mathrm{r}} G : E(b\lambda_s^*)\in K\,\forall s\in G\},
\]
and $K\mathbin{\bar{\rtimes}}_{\mathrm{r}} G$ is a closed ideal in $\rcp{B}{G}$ which contains $\rcp{K}{G}$. (In fact, these two ideals coincide whenever $G$ is exact.) The following lemma is inspired by \cite{AS1994}.

\begin{lem} \label{lem:ideal-inclusions}
Let $G$ be discrete C*-simple group with universal G-boundary $\fb$. Let $A$ be a unital $G$-C*-algebra and let $J$ be a closed ideal in $\rcp{(A\otimes C(\fb))}{G}$. Then setting $J_A= J \cap (A\otimes C(\fb))$,
\[
J_A\rtimes_{\mathrm{r}} G \subset J \subset J_A\mathbin{\bar{\rtimes}}_{\mathrm{r}} G.
\]
\end{lem}

\begin{proof}
For $x\in \fb$, let $J_A^x=(\operatorname{id}_A\otimes\delta_x)(J_A)$, where $\operatorname{id}_A\otimes\delta_x$ is the homomorphism from $A\otimes C(\fb)$ onto $A$ given by evaluation at $x$. Then $J_A^x$ is a (potentially non-proper) ideal of $A$.

Let $\pi_x$ denote the induced homomorphism from $(A\otimes C(\fb))/J_A$ onto $A/J_A^x$. Note that any irreducible representation of $(A\otimes C(\fb))/J_A$ factors through some $\pi_x$, and hence the set $\{\pi_x \mid x \in \fb\}$ is a faithful family of representations.

Let $x\in \fb$ be such that $J_A^x \neq A$. Consider the composition of the map
\begin{align*}
J+A\otimes C(\fb) &\to (J+A\otimes C(\fb))/J \\
&\cong (A\otimes C(\fb))/J_A \stackrel{\pi_x}{\to} A/J_A^x
\end{align*}
with any faithful representation of $A/J_A^x$ into $\B(\cH)$. By Arveson's extension theorem, this extends to a unital completely positive map $\Phi_x$ from $\rcp{(A\otimes C(\fb))}{G}$ into $\B(\cH)$.

We claim that $\Phi_x=\Phi_x\circ E$, where $E$ is the canonical conditional expectation onto $A\otimes C(\fb)$. Since $A\otimes C(\fb)$ is contained in the multiplicative domain of $\Phi_x$ and $\Phi_x(f)=f(x)$ for $f\in C(\fb)$, it suffices to show that $\Phi_x(\lambda_s)=0$ for every $s\in G\setminus\{e\}$. By Theorem \ref{thm:c-star-simplicity-2}, $G$ acts freely on $\fb$. Hence there is $h \in C(\fb)$ such that $h(x)=1$ and $\operatorname{supp}(h)\cap s\operatorname{supp}(h)=\emptyset$. Then $\Phi_x(\lambda_s)=\Phi_x(h\lambda_sh) = \Phi_x(h(sh)\lambda_s) =0$, which proves the claim.

Thus, we see that $\Phi_x(E(J))=\Phi_x(J)=0$ for all $x\in \fb$. Since $E(J)\subset A\otimes C(\fb)$, one obtains $E(J)\subset J_A$, or equivalently that $J\subset J_A\mathbin{\bar{\rtimes}}_{\mathrm{r}} G$. The other inclusion, $\rcp{J_A}{G}\subset J$ is obvious.
\end{proof}

\begin{proof}[Proof of Theorem \ref{thm:simplicity-crossed-products}]
Let $A$ be a unital $G$-C*-algebra and let $I$ be a non-trivial closed  ideal in $\rcp{A}{G}$. By Lemma \ref{lem:ideal-gen-nontriv-ideal}, the ideal $J$ of $\rcp{(A\otimes C(\partial_FG))}{G}$ generated by $I$ is non-trivial, and by Lemma \ref{lem:ideal-inclusions}, for $J_A=J\cap (A\otimes C(\partial_FG))$ we have $J\subset J_A\mathbin{\bar{\rtimes}}_r G$.
It follows that $I_A=J \cap A$ is a proper ideal such that $I\subset I_A\mathbin{\bar{\rtimes}}_r G$.
By the assumption that $A$ has no non-trivial $G$-invariant closed ideal, it follows that $I_A=\{0\}$, and hence that $I=\{0\}$.
\end{proof}

\begin{rem}
Theorem \ref{thm:simplicity-crossed-products} applies in particular when $A=C(X)$ for a minimal compact $G$-space $X$.
\end{rem}

The next result was first established in \cite{KK2014}.

\begin{cor} Let $G$ be a discrete group and $X$ a $G$-boundary. Then $G$ is C*-simple if and only if $C(X) \rtimes_r G$ is simple.
\end{cor}

\begin{proof}
The ``only if'' direction follows from the previous theorem. The ``if'' direction is immediate from Lemma \ref{lem:ideal-gen-nontriv-ideal}.
\end{proof}

\subsection{Stabilizer subgroups} \label{sec:stabilizer-subgroups}

Let $G$ be a discrete group and let $X$ be a $G$-boundary. In this subsection, we relate the C*-simplicity of $G$ to the structure of stabilizer subgroups $G_x$ for $x \in X$. If $X$ is topologically free, then $G$ is C*-simple by Theorem \ref{thm:c-star-simplicity-2}. We note that the converse is not true, even if we assume the $G$-action on $X$ to be faithful (for example, Thompson's group $V$ is C*-simple but its action on the circle is not topologically free \cite{LM2016}). However the following holds:

\begin{prop} Let $G$ be a discrete group and let $X$ be a $G$-boundary. Assume there is $x \in X$ such that the point stabilizer $G_x$ is amenable. Then the following are equivalent:
\begin{enumerate}
\item $G$ is C*-simple,
\item $X$ is topologically free.
\end{enumerate}
\end{prop}

\begin{proof} If $X$ is not topologically free, then Proposition \ref{prop:not-weakly-contained-quasi-regular} shows that $\lambda_{G}$ is not weakly contained in $\lambda_{G/G_x}$. However, $G_x$ is amenable, and hence $G$ is not C*-simple (see Remark \ref{weak-cont-def}). \end{proof}

\begin{rem} The implication $(1) \implies (2)$ holds more generally when $X$ is assumed only to be a minimal compact $G$-space. To see this, note that the amenability of $G_x$ allows us to define a *-representation $\pi_x : C(X) \rtimes_r G \to B(\ell^2(G/G_x))$ by
\begin{align*}
\pi_x(\lambda_s) = \lambda_{G/G_x}(s), &\quad s \in G, \\
\pi_x(f) \delta_p = f(px) \delta_p, &\quad f \in C(X),\ p \in G/G_x.
\end{align*}
This representation cannot be faithful if $X$ is not topologically free, because there will be $s \in G\setminus\{e\}$ and $x \in X$ such that $s$ fixes a neighborhood of $x$, and $\pi_x(\lambda_s f)=\pi_x(f)$ for any function $f \in C(X)$ supported on the fixed points of $s$. However, this contradicts the simplicity of $C(X) \rtimes_r G$ obtained from Theorem \ref{thm:simplicity-crossed-products}.
\end{rem}

When the $G$-boundary $X$ is topologically free, then $G_x$ may or may not be amenable. For example, the action of $\mathrm{PSL}_d(\mathbb{Z})$ on $\mathbb{P}(\IR^d)$ is a topologically free boundary action with $G_x=\{e\}$ for all but countably many $x$, but with $G_x$ non-amenable for some points $x$, when $d \geq 3$. If we assume that $X$ is not topologically free (which is typically the case in practice when one does not know if $G$ is C*-simple), then this proposition shows that if $G$ is C*-simple, then all stabilizer subgroups are non-amenable. The converse assertion is not true, even if we assume that the $G$-action is faithful on $X$ (this relies on Le Boudec's example of a non C*-simple group without amenable radical, see \cite{LM2016}*{Example 3.15}). However, the following proposition provides a partial converse.

\begin{prop} \label{prop:G-not-c-star-simple-implies-G-x-not-c-star-simple}
Let $G$ be a discrete group and let $X$ be a $G$-boundary. If there is $x \in X$ with $G_x$ C*-simple, then $G$ is C*-simple.
\end{prop}

The remainder of this section is devoted to the proof. Let $H \le G$ be a subgroup and let $E_H$ denote the canonical conditional expectation from $\ca_r(G)$ onto $\ca_r(H)$ defined by $E_H(\lambda_s)=\lambda_s$ for $s\in H$ and $E_H(\lambda_s)=0$ for $s \in G \setminus H$. Thus the canonical tracial state $\tau_\lambda$ on $\ca_r(G)$ corresponds to $E_{\{e\}}$.

For every non-trivial closed ideal $I$ of $\ca_r(G)$, the subspace $E_H(I)$ is a (possibly non-closed) non-zero ideal of $\ca_r(H)$. Indeed, $E_H(I)$ is an ideal since $E_H$ is a $\ca_r(H)$-bimodule map, and it is non-zero since $\tau_\lambda = \tau_\lambda \circ E_H$ is faithful on $\ca_r(G)$, i.e. does not vanish on non-zero positive elements.

For $x \in X$, the conditional expectation $E_{G_x}$ extends to a conditional expectation $E_x$ from $\rcp{C(X)}{G}$ onto $\ca_r(G_x)$ satisfying $E_x(f\lambda_s)=f(x)E_{G_x}(\lambda_s)$ for $f\in C(X)$ and $s\in G$. The proposition follows immediately from the following lemma.

\begin{lem}
Let $I$ be a non-trivial closed ideal in $\ca_r(G)$. Then for every $x \in X$, the closure of $E_{G_x}(I)$ is a non-trivial closed ideal of $\ca_r(G_x)$, where $E_{G_x}$ denotes the conditional expectation from $\ca_r(G)$ onto $\ca_r(G_x)$.
\end{lem}

\begin{proof}
Let $x\in X$ be given. It suffices to show $E_{G_x}(I)$ is not dense in $\ca_r(G_x)$.
By Lemma \ref{lem:ideal-gen-nontriv-ideal}, the closed ideal
$J$ of $C(X)\rtimes_{\mathrm{r}} G$ generated by $I$ is a proper ideal
which has zero intersection with $C(X)$. As usual, we consider the state on $C(X) + J$ obtained by composing the map
\[
C(X)+J \to (C(X)+J)/J \cong C(X)
\]
with point evaluation at $x$. Let $\phi_x$ be a state extension on $\rcp{C(X)}{G}$.

Since $C(X)$ is contained in the multiplicative domain (see $(\ref{mult-dom})$) of $\phi_x$, we have $\phi_x(f\lambda_s)=f(x)\phi_x(\lambda_s)$ for every $f\in C(X)$ and $s\in G$.  We claim that $\phi_x(\lambda_s)=0$ for every $s\in G\setminus G_x$. Indeed, there is $h \in C(X)$ such that $h(x)=1$ and $\operatorname{supp}(h)\cap s\operatorname{supp}(h)=\emptyset$,
and hence $\phi_x(\lambda_s)=\phi_x(h\lambda_sh)=\phi_x(h(sh)\lambda_s)=0$. It follows that $\phi_x=\phi_x\circ E_x$. Since $\phi_x(I)=0$, the ideal $E_{G_x}(I)=E_x(I)$ is not dense.
\end{proof}

\begin{rem} The recent preprint \cite{LM2016}*{Corollary 3.14} proves another result of this kind: if the boundary $X$ is \emph{extreme}, in the sense that every proper closed subset can be sent into any given open subset by a group element, then the existence of a point $x$ with non-amenable neighborhood stabilizer $G_x^0$ (i.e. the subgroup of elements a neighborhood of $x$ pointwise) implies the C*-simplicity of $G$.
\end{rem}

\section{Connes-Sullivan conjecture} \label{sec:connes-sullivan}

In this section we discuss a new operator algebraic property of discrete groups, which is stronger than C*-simplicity and still holds for many of the examples of Section \ref{sec:cohomology} and Section \ref{sec:linear}. It is connected to the following theorem of Zimmer \cite{Z1987}, which had been conjectured by Connes and Sullivan.

\begin{csconj}
Let $L$ be a connected Lie group, $H\le L$  a closed subgroup, and $m$ a $L$-quasi-invariant Borel measure on $L/H$. If $G$ is an abstract subgroup of $L$, whose action on $(L/H,m)$ is amenable in the sense of Zimmer, where $G$ is regarded as a discrete group, then $(\overline{G})^0$ is solvable, where $(\overline{G})^0$ denotes the connected component of the identity of the closure $\overline{G}$ of $G$ in $L$.
\end{csconj}

This conjecture was proved by Carri\`ere and Ghys \cite{cg} for the case when $G$ is a free group, and by Zimmer \cite{Z1987} in general. Breuillard and Gelander \cites{bg1,bg2} gave an alternative proof along the lines of Carri\`ere and Ghys. We will give a new interpretation of the Connes--Sullivan conjecture and generalize it.

Recall that if a non-singular action $G\curvearrowright (X,m)$ is \emph{amenable}, then the Koopman unitary representation $\pi$ of $G$ on $L^2(X,m)$ is \emph{weakly regular} \cites{kuhn,delaroche}. (For the purposes of this paper, this is essentially all that one has to know about the amenability of a non-singular action.) A unitary representation $\pi$ of $G$ on $\cH$ is said to be weakly regular if it is weakly contained in the regular representation $\lambda$ of $G$ on $\ell^2(G)$, i.e., if $\|\pi(f)\|_{\IB(\cH)} \le \|\lambda(f)\|_{\IB(\ell^2(G))}$ for every $f\in\IC[G]$. We observe that if $(X,m)=(L/H,m)$, then the Koopman representation $\pi\colon G\to\IB(\cH)$ extends to a \emph{continuous} unitary representation of the ambient group $L$. Here the continuity is with respect to the \emph{strong operator topology} (SOT, in short). A subbase of open sets for the SOT is given by $U(S,\xi,\varepsilon):=\{ T \in\IB(\cH) : \| (T - S) \xi \|_{\cH}<\varepsilon\}$ for $S\in\IB(\cH)$, $\xi\in\cH$, and $\varepsilon>0$. Therefore, Zimmer's theorem (the Connes--Sullivan conjecture) follows (see Remark \ref{CS-proof} below) if $G$ as above has the property (CS) defined below. We recall that the \emph{amenable radical} $R_a(G)$ of $G$ is the largest amenable normal subgroup of $G$.

\begin{defn}
We say a discrete group $G$ has property (CS) if the following statement holds: For every weakly regular unitary representation $\pi\colon G\to\IB(\cH)$, there exists an SOT-neighborhood $U$ of the identity in $\IB(\cH)$ such that $\pi^{-1}(U) \subset R_a(G)$.
\end{defn}

\begin{prop} Every discrete group $G$ with property (CS) and trivial amenable radical is C*-simple.
\end{prop}

\begin{proof} We will prove that (CS) implies the absence of amenable normalish subgroups. The proposition will then follow from Theorem \ref{thm:not-c-star-simple-implies-normalish}. If $H \le G$ is amenable, then the quasi-regular representation $\pi=\lambda_{G/H}$ of $G$ on $\ell^2(G/H)$ is weakly regular, but if $H$ is normalish, then $\pi(G)$ is clearly not discrete.
\end{proof}

\begin{rem}The converse is not true. The Baumslag-Solitar groups $G = BS(m,n)$ from Section \ref{sec:baumslag-solitar} are C*-simple, but do not have property (CS). In fact, the subgroup $\Lambda:=\ip{a} < G$ is not only amenable and normalish, but also commensurated in $G$, i.e., for every $t\in G$, the subgroup $\Lambda\cap t\Lambda t^{-1}$ has finite index in $\Lambda$. This implies that the Schlichting completion $\overline{G}$ (which is the closure of $G$ in $\mathrm{Sym}(G/\Lambda)$ with respect to the pointwise convergence topology) is a locally compact group, and the action of $G$ on $(\overline{G},m)$ is amenable (because it has $G/\Lambda$ as a factor).
\end{rem}

As we have explained above, the following theorem generalizes Zimmer's theorem. The proof relies on the Strong Tits Alternative \cite{breuillard} and the fact from Theorem \ref{thm:linear-groups-non-amenability-normalish} that a linear group with a trivial amenable radical has no amenable \emph{normalish} subgroups.

\begin{thm}\label{thm:linear}
Every linear group $G$ has property (CS).
\end{thm}

\begin{proof}
Let $G\le\mathrm{GL}_d(K)$ be any linear group and let $\pi\colon G\to\IB(\cH)$ be a weakly regular unitary representation of $G$. Fix a unit vector $\xi\in\cH$ and let $U_0:=U(1,\xi,\varepsilon)=\{ T\in\IB(\cH) : \|(T-1)\xi\|<\varepsilon\}$. We claim that there is $\varepsilon>0$ (in fact $\varepsilon$ will depend only on $d$, and not on $K$, $G$, $\pi$, nor $\xi$) such that $\pi^{-1}(U_0)$ generates an amenable subgroup of $G$.

By the Strong Tits Alternative \cite{breuillard}*{Theorem 1.1}, there is $N=N(d)\in\IN$ such that for any symmetric subset $S\subset G$ containing $1$, either $S^N$ contains two elements $\{ g_1,g_2\}$ which freely generate a non abelian free group or the subgroup $\ip{S}$ generated by $S$ is amenable. (Note that by the Tits Alternative, the group $\ip{S}$ is amenable if and only if every finitely generated subgroup of it is virtually solvable.)

Suppose the former possibility occurs and let $h:=(g_1+g_1^{-1}+g_2+g_2^{-1})/4\in\IC[G]$. By Kesten's theorem, one has $\|\pi(h)\|\le\|\lambda(h)\|=\sqrt{3}/2$. It follows that for any unit vector $\xi\in\cH$, one has
\[
\|\pi(g_1)\xi-\xi\|^2+\|\pi(g_2)\xi-\xi\|^2=4(1-\ip{\pi(h)\xi,\xi})\geq 4-2\sqrt{3}.
\]
Hence, $\|\pi(g_i)\xi-\xi\|\geq\sqrt{2-\sqrt{3}}=(\sqrt{6}-\sqrt{2})/2$ for some $i\in\{1,2\}$, and so $\|\pi(s)\xi-\xi\|\geq(\sqrt{6}-\sqrt{2})/2N$ for some $s\in S$. This means that $\varepsilon=(\sqrt{6}-\sqrt{2})/2N$ does the job, and the claim is proved.

Now, since $\Lambda:=\ip{\pi^{-1}(U_0)}$ is amenable, by Theorem \ref{thm:linear-groups-non-amenability-normalish} one may find a finite subset $F\subset G$ such that $\bigcap_{t\in F} t \Lambda t^{-1} \subset R_a(G)$. Hence, for $U=\bigcap_{t\in F} \pi(t)U_0\pi(t)^*$, one has $\pi^{-1}(U) \subset \bigcap_{t\in F} t \Lambda t^{-1} \subset R_a(G)$.
\end{proof}

\begin{rem}\label{CS-proof} To see that this implies indeed Zimmer's theorem, note that (CS) implies that there is a neighborhood of the identity $U$ in the Lie group $L$ such that the subgroup generated by $G \cap U$ is amenable. However every neighborhood of the identity in a connected Lie group generates the Lie group. It is then not difficult to see that $G \cap (\overline{G})^0 \cap U$ generates $G \cap (\overline{G})^0$, which is thus amenable, and therefore virtually solvable. As a connected Lie group $(\overline{G})^0$ has no subgroup of finite index, so it must be solvable.
\end{rem}

There are many more examples of groups with the property (CS). Recall the definition of the $n$-th bounded cohomology group of $G$ from Section \ref{sec:cohomology}. Hull and Osin (\cites{ho,O2013}) proved that if $G/R_a(G)$ is acylindrically hyperbolic, then $H_{\mathrm{b}}^2(G,\ell^2(G/R_a(G)))\neq0$. Thus, the next result implies in particular that acylindrically hyperbolic groups have property (CS).

\begin{thm}\label{thm:hyp}
Every group $G$ such that $H_{\mathrm{b}}^2(G,\ell^2(G/R_a(G)))\neq0$ has property (CS).
\end{thm}

Before presenting a proof of Theorem \ref{thm:hyp} in full generality, we will first present a proof for when $G$ is the free group $F_r$ of rank $r$, since the proof in this special case is particularly easy and illuminating.

\begin{proof}[Proof of Theorem \ref{thm:hyp} for $G=F_r$]
We identify $F_r$ with the Cayley graph with respect to the free generators, which is a regular tree of degree $2r$. Let $X$ be the visual boundary of $F_r$. Every triplet $\{x_1,x_2,x_3\}$ of distinct points in $X$ determines the center $c(x_1,x_2,x_3)\in F_r$ of the tripod spanned by $\{x_1,x_2,x_3\}$. Namely, $c(x_1,x_2,x_3)$ is the unique point in $F_r$ which belongs to the intersection $[x_1,x_2]\cap[x_2,x_3]\cap[x_3,x_1]$ of the geodesic paths $[x_i,x_j]$ connecting distinct points $x_i$ and $x_j$. We note that $c(sx_1,sx_2,sx_3)=sc(x_1,x_2,x_3)$ for every $s\in F_r$ and every $\{x_1,x_2,x_3\}$.

Take open subsets $U_i$, $i=1,2,3$, in $X$ such that $c(x_1,x_2,x_3)=1$ for all $(x_1,x_2,x_3)\in U_1\times U_2\times U_3$. Since $G$ acts minimally on $X$, for each $i$ there is a finite subset $F_i\subset G$ such that $X=F_iU_i$. On the other hand, by the equivariance of the center map,
\[
s(U_1 \times U_2 \times U_3) \cap (U_1 \times U_2 \times U_3) = \emptyset,
\]
for every $s \in G \setminus \{e\}$, where $G$ acts diagonally on $X^3$. Now pick $z\in X$ and put $A_i=\{ s\in G : sz \in U_i\}$. By the first condition, $A_i$ is left syndetic, i.e., $G=F_iA_i$. By the second condition,
\[
s(A_1 \times A_2 \times A_3) \cap (A_1 \times A_2 \times A_3) = \emptyset
\]
for every $s \in G \setminus \{e\}$.

Let $\pi\colon G\to\IB(\cH)$ be a weakly regular unitary representation. Then, the identity map on $\IC[G]$ extends to a continuous $*$-homomorphism $\mathrm{C}^*_\lambda(G)\ni\lambda(f)\mapsto\pi(f)\in\mathrm{C}^*_\pi(G)$. By Arveson's Extension Theorem (see \cite{bo}*{Theorem 1.6.1}), this map extends to a unital completely positive map $\varphi$ from $\IB(\ell^2(G))$ into $\IB(\cH)$. Furthermore, since $\ca_r(G)$ belongs to the multiplicative domain of $\varphi$, $\varphi(\lambda(s)T\lambda(t))=\pi(s)\varphi(T)\pi(t)$ for every $T\in\IB(\ell^2(G))$ and $s,t\in G$ (see \cite{bo}*{Proposition 1.5.6}).

We identify $\ell^\infty(G) \subset \IB(\ell^2(G))$ with the subalgebra of diagonal operators and let $a_i:=\varphi(\chi_{A_i})\in\IB(\cH)$, where $\chi_{A_i}$ denotes the characteristic function corresponding to $A_i$. Since
\[
\sum_{s\in F_i} \pi(s)a_i\pi(s)^*=\varphi \left( \sum_{s\in F_i}\lambda(s)\chi_{A_i}\lambda(s)^* \right)
=\varphi \left( \sum_{s\in F_i}\chi_{sA_i} \right) \geq 1,
\]
each $a_i$ is a non-zero positive operator. Thus $a=a_1\otimes a_2\otimes a_3\in\IB(\cH^{\otimes3})$ is a non-zero positive operator satisfying
\begin{align*}
\sum_{s\in G} \pi^{\otimes 3}(s)a\pi^{\otimes 3}(s)^*
 &= \sup_{E\subset G\,\mathrm{finite}}\varphi \left(\sum_{s\in E}
  \lambda^{\otimes 3}(s)\chi_{A_1\times A_2\times A_3}\lambda^{\otimes 3}(s)^* \right)\\
 &=\sup_{E\subset G\,\mathrm{finite}}\varphi \left(\sum_{s\in E}\chi_{s(A_1\times A_2\times A_3)} \right)\\
 &\le 1.
\end{align*}
Hence there cannot be a sequence $(s_n)_n$ in $G$ such that $\pi(s_n)\to1$ in the SOT.

In fact, a stronger statement holds. For any unit vector $\xi\in\cH^{\otimes3}$ with corresponding orthogonal projection denoted by $P_\xi$, one has
\begin{align*}
\sum_{s\in G} |\ip{\pi^{\otimes 3}(s) a^{1/2}\xi , a^{1/2}\xi}|^2
 &= \sum_{s\in G} \ip{\pi^{\otimes 3}(s) a^{1/2}P_\xi a^{1/2}\pi^{\otimes 3}(s)^* a^{1/2}\xi,a^{1/2}\xi}\\
 &\le \ip{ a^{1/2}\xi,a^{1/2}\xi}\\
 &\le 1.
\end{align*}
This means that $\pi^{\otimes 3}$ contains a non-zero subrepresentation which is unitarily equivalent to a subrepresentation of the regular representation $\lambda$.
\end{proof}

\begin{proof}[Proof of Theorem \ref{thm:hyp}]
Let $V$ be a coefficient $G$-module, let $V_*$ denote the predual of $V$ and let $\pi\colon G\to\IB(\cH)$ be an weakly regular unitary representation of $G$. The group $G$ acts isometrically on $\IB(\cH)$ by conjugation. Noting that $\ell^\infty(G^n,V)=\IB(V_*,\ell^\infty(G^n))$, we consider the complex
\[
0\longrightarrow \IB(V_*,\IB(\cH))^G\stackrel{d_1'}{\longrightarrow}
\IB(V_*,\IB(\cH^{\otimes 2}))^G\stackrel{d_2'}{\longrightarrow}
\IB(V_*,\IB(\cH^{\otimes 3}))^G\stackrel{d_3'}{\longrightarrow}\cdots,
\]
where $G$ acts on $\IB(V_*,\IB(\cH^{\otimes n}))$ by $(s\cdot F)(v_*)=\pi^{\otimes n}(s)F(s^{-1}v_*)\pi^{\otimes n}(s)^*$ and $(d_n' F)(v_*)=\partial_n (F(v_*))$, where $\partial_n\colon\IB(\cH^{\otimes n})\to\IB(\cH^{\otimes(n+1)})$ is defined by
\[
\partial_n(T) = \sum_{j=0}^{n}(-1)^j W_j^*(1\otimes T)W_j,
\]
where $W_j$ denotes the unitary operator associated with the permutation $(j\ 0\ \cdots\  j-1\ j+1\ \cdots\  n)$. We claim that $H_{\mathrm{b}}^n(G,V)$ embeds into $\ker d'_{n+1} / \ran d'_n$. Before proving the claim, we first explain how this finishes the proof.

Since $H_{\mathrm{b}}^2(G,V)\neq0$ for $V=\ell^2(G/R_a(G))$,
one has $\IB(V_*,\IB(\cH^{\otimes 3}))^G\neq0$ by the claim. Take a non-zero $F\in\IB(V_*,\IB(\cH^{\otimes 3}))^G$ and vectors $v_*\in V_*$ and $\xi\in\cH$ such that $\ip{F(v_*)\xi,\xi}=1$. Let $U_0:=U(1,\xi,(4\|F(v_*)\|\|\xi\|)^{-1})$. Then, for every $s\in\pi^{-1}(U_0)$, one has
\[
|\ip{F(sv_*)\xi,\xi}|=|\ip{\pi(s)F(v_*)\pi(s)^*\xi,\xi}|\geq1/2.
\]
Let $Q\colon G\to G/R_a(G)$ be the quotient map. Since $sv_*\to0$ weakly as $Q(s)\to\infty$, the subset $Q( \pi^{-1}(U_0) )$ is finite. Since every nontrivial conjugacy class of $G/R_a(G)$ is infinite \cite{O2013}*{Theorem 8.4}, there is a finite subset $F\subset G$ such that $\bigcap_{t\in F} t\pi^{-1}(U_0) t^{-1}\subset R_a(G)$. Hence, for $U=\bigcap_{t\in F} \pi(t)U_0\pi(t)^*$, one has $\pi^{-1}(U) \subset R_a(G)$.

For the proof of the claim, we recall the following general fact about completely positive maps: For any Hilbert spaces $\cK_1,\cK_2,\cL$ and any unital completely positive map
$\theta\colon\IB(\cK_1)\to\IB(\cK_2)$, there is a unique unital completely positive map $\theta\otimes\id\colon \IB(\cK_1\otimes\cL)\to \IB(\cK_2\otimes\cL)$  that satisfies $(\id\otimes P_{\cL_0})(\theta\otimes\id) = (\theta\otimes\id)(\id\otimes P_{\cL_0})$ for compressions $P_{\cL_0}\colon\IB(\cL)\to\IB(\cL_0)$ with respect to finite dimensional subspaces $\cL_0\subset\cL$. Indeed, if we make the identification $\cL=\ell^2(I)$, then elements $T$ in $\IB(\cK_1\otimes\cL)$ can be viewed as $I\times I$ matrices $[T_{i,j}]_{i,j\in I}$ with entries $T_{i,j}$ in $\IB(\cK_1)$ such that
\[
\| T \|=\sup_{I_0 \subset I\,\mathrm{finite}}\| [T_{i,j}]_{i,j\in I_0}\|_{\IB(\cK_1\otimes\ell^2(I_0))}<\infty.
\]
Under this identification, the unital completely positive map $\theta\otimes\id$ is defined by $(\theta\otimes\id)([T_{i,j}]_{i,j})=[\theta(T_{i,j})]_{i,j}$. We can similarly define $\id\otimes\theta'$, and $\theta\otimes\id$ and $\id\otimes\theta'$ will commute if at least one of $\theta$ or $\theta'$ is weak*-continuous. Now take unital completely positive maps $\Phi\colon\IB(\ell^2(G))\to\IB(\cH)$ and $\Psi\colon\IB(\cH)\to\IB(\ell^2(G))$ which are $G$-equivariant. (The existence of such maps follows from Arveson's Extension Theorem as in the Proof of Theorem \ref{thm:hyp} for $G=F_r$.)

We look at $\varphi=\Phi|_{\ell^\infty(G)}$ and $\psi=E\circ\Psi$, where $E$ is the completely positive projection from $\IB(\ell^2(G))$ onto the diagonal $\ell^\infty(G)$. We define the $G$-equivariant unital completely positive maps $\varphi_n\colon\ell^\infty(G^{n+1})\to\IB(\cH^{\otimes(n+1)})$ by $\varphi_0=\varphi$ and $\varphi_n=(\varphi_{n-1}\otimes\id)(\id\otimes\varphi)$. We similarly define $\psi_n\colon\IB(\cH^{\otimes(n+1)})\to\ell^\infty(G^{n+1})$.

The map $\varphi_n$ extends to a map from $\ell^\infty(G^{n+1},V)$ to $\IB(V_*,\IB(H^{\otimes(n+1)}))$ by identifying $\ell^\infty(G^{n+1},V)$ with $\IB(V_*,\ell^\infty(G^{n+1}))$ and composing $\varphi_n$ with $F \in \IB(V_*, \ell^\infty(G^{n+1}))$. We continue to denote this extension by $\varphi_n$. Similarly, $\psi_n$ extends to a map from $\IB(V_*,\IB(H^{\otimes (n+1)}))$ to $\ell^\infty(G^{n+1},V)$ that we continue to denote by $\psi_n$.

These maps then become morphisms of complexes. Indeed, let us verify that $d'_n\varphi_{n-1}=\varphi_n d_n$ by induction on $n$. The case $n=1$ is obvious, and letting $\iota_{n+1}(T) = T\otimes1$ gives
\begin{align*}
d'_{n+1}\varphi_{n}
 &=(d'_{n}\otimes\id +(-1)^{n+1}\iota_{n+1})\varphi_n\\
 &=(\varphi_n d_n \otimes\id)(\id\otimes\varphi)+(-1)^{n+1}(\varphi_n\otimes1)\iota_{n+1}\\
 &=\varphi_{n+1}(d_n\otimes\id+(-1)^{n+1}\iota_{n+1})\\
 &=\varphi_{n+1}d_{n+1}.
\end{align*}
Hence, $(\psi_n\varphi_n)_n$ is a $G$-equivariant morphism on the complex defining the bounded cohomology. As such, it induces the identity maps on the cohomology groups by Lemma~7.2.6 in \cite{M2001}. It follows that $H_{\mathrm{b}}^n(G,V)$ embeds into $\ker d'_{n+1} / \ran d'_n$, proving the claim.
\end{proof}



\begin{thebibliography}{99}

\bib{AGV2014}{article}{
author={Ab{\'e}rt, M.},
author={Glasner, Y.},
author={Vir{\'a}g, B.},
title={Kesten's theorem for invariant random subgroups},
journal={Duke Mathematical Journal},
volume={163},
pages={465-488},
year={2014}
}

\bib{A1983}{article}{
author={Adyan, S.I.},
title={Random walks on free periodic groups},
journal={Izv. Math.},
volume={21},
pages={425-434},
year={1983}
}

\bib{AS1994}{article}{
author={Archbold, R.J.},
author={Spielberg, J.S.},
title={Topologically free actions and ideals in discrete C*-dynamical systems},
journal={Proc. of the Edinburgh Math. Soc.},
volume={37},
date={1994},
pages={119-124}
}

\bib{delaroche}{article}{
author={Anantharaman-Delaroche, C.},
title={On spectral characterizations of amenability},
journal={Israel Journal of Mathematics},
volume={137},
year={2003},
pages={1-33}
}

\bib{BDL2014}{article}{
author={Bader, U.},
author={Duchesne, B.},
author={Lecureux, J.},
title={Amenable Invariant Random Subgroups},
journal={arXiv preprint arXiv:1409.4745},
year={2014}
}

\bib{BFS2013}{article}{
author={Bader, U.},
author={Furman, A.},
author={Sauer, R.},
title={Weak notions of normality and vanishing up to rank in $L^2$-cohomology},
journal={to appear in International Mathematical Research Notices}
}

\bib{BCH1994}{article}{
author={Bekka, M.},
author={Cowling, M.},
author={de la Harpe, P.},
title={Some groups whose reduced C*-algebra is simple},
journal={Inst. Hautes \'Etudes Sci. Publ. Math.},
number={80},
date={1994},
pages={117-134},
}

\bib{BJ2014}{article}{
author={Bleak, C.},
author={Juschenko, K.},
title={Ideal structure of the C*-algebra of Thompson group T},
journal={arXiv preprint arXiv:1409.8099},
year={2014}
}

\bib{BK2016}{article}{
author={Bryder, R.S.},
author={Kennedy, M.},
title={Reduced twisted crossed products over C*-simple groups},
journal={arXiv preprint arXiv:1602.01533},
year={2016}
}

\bib{breuillard}{article}{
author={Breuillard, E.},
title={A strong Tits alternative},
journal={arXiv preprint arXiv:0804.1395},
year={2013}
}

\bib{bg1}{article}{
author={Breuillard, E.},
author={Gelander, T.},
title={A topological Tits alternative},
journal={Annals of Mathematics},
volume={166},
year={2007},
pages={427-474}
}

\bib{bg2}{article}{
author={Breuillard, E.},
author={Gelander, T.},
title={Uniform independence in linear groups},
journal={Invent. Math. },
volume={173},
year={2008},
pages={225-263}
}

\bib{bo}{book}{
author={Brown, N.},
author={Ozawa, N.},
title={C*-algebras and Finite-Dimensional Approximations},
series={Graduate Studies in Mathematics},
publisher={American Math. Soc.},
volume={88},
year={2008},
place={Providence, RI.}
}

\bib{cg}{article}{
author={Carri\`ere, Y.},
author={Ghys, \'E.},
title={Uniform independence in linear groups},
journal={Comptes Rendus de l'Acad\'emie des Sciences},
volume={300},
year={1985},
pages={677-680}
}

\bib{CE1977}{article}{
author={Choi, M. D.},
author={Effros, E. G.},
title={Injectivity and operator spaces},
journal={Journal of Functional Analysis},
volume={24},
year={1977},
pages={156-209}
}

\bib{DGO2011}{article}{
author={Dahmani, F.},
author={Guirardel, V.},
author={Osin, D.},
title={Hyperbolically embedded subgroups and rotating families in groups acting on hyperbolic spaces},
journal={arXiv preprint arXiv:1111.7048},
year={2011}
}
\bib{D1957}{article}{
author={Day, M.},
title={Amenable semigroups},
journal={Illinois Journal of Mathematics},
volume={1},
year={1957},
pages={459-606}
}

\bib{D2007}{article}{
author={de la Harpe, P.},
title={On simplicity of reduced C*-algebras of groups},
journal={Bulletin of the London Mathematical Society},
volume={39},
year={2007},
pages={1-26}
}

\bib{DP2011}{article}{
author={de la Harpe, P.},
author={Pr{\'e}aux, J.P.},
title={C*-simple groups: amalgamated free products, HNN extensions, and fundamental groups of 3-manifolds},
journal={Journal of Topology and Analysis},
volume={3},
pages={451-489},
year={2011}
}

\bib{DS1986}{article}{
author={de La Harpe, P.},
author={Skandalis, G.},
title={Powers' property and simple C*-algebras},
journal={Mathematische Annalen},
volume={273},
pages={241-250},
year={1986}
}

\bib{Di1977}{book}{
author={Dixmier, J.},
title={C*-algebras},
publisher={North-Holland Publishing Co.},
year= {1977}
}

\bib{F1971}{article}{
author={Frol{\'\i}k, Z.},
title={Maps of extremally disconnected spaces, theory of types, and applications},
journal={General Topology and its Relations to Modern Analysis and Algebra},
pages={131-142},
year={1971},
publisher={Academia Publishing House of the Czechoslovak Academy of Sciences}
}

\bib{F2003}{article}{
author={Furman, A.},
title={On minimal strongly proximal actions of locally compact groups},
journal={Israel Journal of Mathematics},
volume={136},
year={2003},
pages={173-187}
}

\bib{F1973}{book}{
author={Furstenberg, H.},
title={Boundary theory and stochastic processes on homogeneous spaces},
series={Harmonic analysis on homogeneous spaces, Proceedings of the Symposium of Pure Mathematics},
publisher={American Mathematical Society},
volume={26},
year={1973},
place={Providence, RI.}
}

\bib{G2000}{article}{
author={Gaboriau, D.},
title={Co{\^u}t des relations d'{\'e}quivalence et des groupes},
journal={Invent. Math.},
volume={139},
pages={41-98},
year={2000}
}

\bib{G1974}{article}{
author={Glasner, S.},
title={Topological dynamics and group theory},
journal={Transactions of the American Mathematical Society},
volume={187},
pages={327-334},
year={1974}
}

\bib{G1976}{book}{
author={Glasner, S.},
title={Proximal flows},
series={Lectures Notes in Mathematics},
publisher={Springer},
volume={517},
year={1976},
}

\bib{gleason}{article}{
author={Gleason, A. M.},
title={Projective topological spaces},
journal={Illinois Journal of Mathematics},
volume={2},
year={1958},
pages={482-489}
}

\bib{H2015}{article}{
author={Haagerup, U.},
title={A new look at C*-simplicity and the unique trace property of a group},
status={arXiv preprint arXiv:1509.05880},
year={2015}
}

\bib{HO2014}{article}{
author={Haagerup, U.},
author={Olesen, K.K.},
title={On conditions towards the non-amenability of Richard ThompsonÕs group F},
status={in preparation},
year={2014}
}

\bib{H1985}{article}{
author={Hamana, M.},
title={Injective envelopes of C*-dynamical systems},
journal={Tohoku Mathematical Journal},
volume={37},
year={1985},
pages={463-487}
}

\bib{ho}{article}{
author={Hull, M.},
author={Osin. D.},
title={Induced quasicocycles on groups with hyperbolically embedded subgroups},
journal={Algebraic \& Geometric Topology},
volume={13},
year={2013},
pages={2635-2665}
}

\bib{I1994}{article}{
author={Ivanov, S.V.},
title={The free Burnside groups of sufficiently large exponents},
journal={International Journal of Algebra and Computation},
volume={4},
pages={1-308},
year={1994}
}

\bib{IO2016}{article}{
author={Ivanov, N. A.},
author={Omland, T},
title={C*-simplicity of free products with amalgamation and radical classes of groups},
journal={arXiv preprint arXiv:1605.06395},
year={2016}
}

\bib{KT1990}{article}{
author={Kawamura, S.},
author={Tomiyama, J.},
title={Properties of topological dynamical systems and corresponding C*-algebras},
journal={Tokyo Journal of Mathematics},
volume={13},
pages={215-257}
}



\bib{K2015}{article}{
author={Kennedy, M.},
title={Characterizations of C*-simplicity},
journal={arXiv preprint arXiv:1509.01870},
year={2015}
}

\bib{KK2014}{article}{
author={Kalantar, M.},
author={Kennedy, M.},
title={Boundaries of reduced C*-algebras of discrete groups},
journal={Journal f\"{u}r die reine und angewandte Mathematik},
status={to appear},
year={2014}
}

\bib{kuhn}{article}{
author={Kuhn, M.G.},
title={Amenable actions and weak containment of certain representations of discrete groups},
journal={Proc. of the American Math. Soc.},
volume={122},
year={1994},
pages={751-757}
}

\bib{L2015}{article}{
author={Le Boudec, A.},
title={C*-simplicity and the amenable radical},
journal={arXiv:1507.03452},
year={2015}
}


\bib{LM2016}{article}{
author={Le Boudec, A.},
author={Matte Bon, N.},
title={Subgroup dynamics and C*-simplicity of groups of homeomorphisms},
journal={arXiv preprint arXiv:1605.01651},
year={2016}
}

\bib{L1998}{article}{
author={L{\"u}ck, W.},
title={Dimension theory of arbitrary modules over finite von Neumann algebras and $L^2$-Betti numbers I: Foundations},
journal={Journal f\"{u}r die reine und angewandte Mathematik (2014)},
volume={495},
pages={135-162},
year={1998}
}

\bib{MO2013}{article}{
author={Minasyan, A.},
author={Osin, D.},
title={Acylindrical hyperbolicity of groups acting on trees},
journal={arXiv preprint arXiv:1310.6289},
year={2013}
}

\bib{M2001}{book}{
author={Monod, N.},
title={Continuous bounded cohomology of locally compact groups},
year={2001},
publisher={Springer}
}

\bib{MS2006}{article}{
author={Monod, N.},
author={Shalom, Y.},
title={Orbit equivalence rigidity and bounded cohomology},
journal={Annals of mathematics},
pages={825-878},
year={2006},
}

\bib{O1980}{article}{
author={Olshanskii, A.Y.},
title={On the question of the existence of an invariant mean on a group},
journal={Uspekhi Matematicheskikh Nauk},
volume={35},
year={1980},
pages={199-200}
}

\bib{O1991}{book}{
title={Geometry of defining relations in groups},
author={OlÕshanskii, A.Y.},
year={1991},
publisher={Springer}
}

\bib{OO2014}{article}{
author={Olshanskii, A.Y.},
author={Osin, D.V.},
title={C*-simple groups without free subgroups},
year={2014},
journal={arXiv preprint arXiv:1401.7300}
}

\bib{O2013}{article}{
author={Osin, D.V.},
title={Acylindrically hyperbolic groups},
journal={Transactions of the American Mathematical Society},
status={to appear}
}

\bib{paulsen}{book}{
author={Paulsen, V.},
title={Completely bounded maps and operator algebras},
series={Cambridge Studies in Advanced Mathematics},
volume={78},
publisher={Cambridge University Press},
year={2002}
}

\bib{pedersen}{book}{
author={Pedersen, G.K.},
title={C*-algebras and their automorphism groups},
series={London Mathematical Society Monographs},
volume={14},
publisher={Academic Press},
year={1979}
}

\bib{PT2011}{article}{
author={Peterson, J.},
author={Thom, A.},
title={Group cocycles and the ring of affiliated operators},
journal={Invent. Math.},
volume={185},
pages={561-592},
year={2011}
}

\bib{P2012}{article}{
author={Pitts, D.},
title={Structure for Regular Inclusions},
journal={arXiv preprint arXiv:1202.6413},
year={2012}
}

\bib{P1975}{article}{
author={Powers, R.},
title={Simplicity of the C*-algebra associated with the free group on two generators},
journal={Duke Mathematical Journal},
volume={42},
pages={151-156},
year={1975},
publisher={Duke University Press}
}

\bib{P2008}{article}{
author={Poznansky, T.},
title={Characterization of linear groups whose reduced C*-algebras are simple},
journal={arXiv preprint arXiv:0812.2486},
year={2008}
}

\bib{R2015a}{article}{
author={Raum, S.},
title={C*-simplicity of locally compact Powers groups},
journal={Journal f\"{u}r die reine und angewandte Mathematik},
status={to appear},
year={2015}
}

\bib{R2015b}{article}{
author={Raum, S.},
title={Cocompact amenable closed subgroups: weakly inequivalent representations in the left-regular representation},
journal={International Mathematical Research Notices},
year={2015}
}

\bib{T2009}{article}{
author={Thom, A.},
title={Low degree bounded cohomology and $L^2$-invariants for negatively curved groups},
journal={Groups, Geometry and Dynamics},
volume={3},
year={2009},
pages={343-358}
}

\bib{Ti1972}{article}{
author={Tits, J.},
title={Free subgroups in linear groups},
journal={Journal of Algebra},
volume={20},
year={1972},
pages={250-270}
}


\bib{T2012}{article}{
author={Tucker-Drob, R.D.},
title={Shift-minimal groups, fixed price 1, and the unique trace property},
journal={arXiv preprint arXiv:1211.6395},
year={2012}
}

\bib{Z1987}{article}{
author={Zimmer, R.J.},
title={Amenable actions and dense subgroups of Lie groups},
journal={Journal of Functional Analysis},
volume={72},
year={1987},
pages={58-64}
}

\end{thebibliography}
\end{document}